\documentclass{amsart}
\usepackage{amssymb}
\usepackage{mathrsfs}
\usepackage{amsmath}
\usepackage{graphicx}
\usepackage{tikz, subcaption, mathtools, soul, mathdots, yhmath}
\usetikzlibrary{cd}
\usepackage{hyperref}
\usepackage[all]{xy}
\usepackage{bbding} 
\usepackage[normalem]{ulem}

\usepackage[margin=1in]{geometry}

\definecolor{background}{RGB}{230,230,230} 
\definecolor{arc1}{RGB}{0,128,128} 
\definecolor{arc2}{RGB}{210,105,30}	
\definecolor{arc3}{RGB}{80,0,80}		
\definecolor{arc4}{RGB}{34,139,34}		

\newcommand{\Z}{\mathbb{Z}}
\newcommand{\C}{\mathbb{C}}

\newcommand{\cC}{\mathcal{C}}
\newcommand{\cF}{\mathcal{F}}
\newcommand{\cB}{\mathcal{B}}

\newcommand{\cT}{\mathcal{T}}
\newcommand{\cW}{\mathcal{W}}

\tikzstyle{arrow} = [thick,->,>=stealth]

\tikzstyle{mathy}=[nodes={%
   execute at begin node=$,%
   execute at end node=$%
  }]

\theoremstyle{plain}
\newtheorem{theorem}{Theorem}[subsection]
\newtheorem*{thm}{Theorem}
\newtheorem{lemma}[theorem]{Lemma}
\newtheorem{proposition}[theorem]{Proposition}
\newtheorem{corollary}[theorem]{Corollary}

\newtheorem{innercustomthm}{Theorem}
\newenvironment{customthm}[1]{\renewcommand\theinnercustomthm{#1}\innercustomthm}{\endinnercustomthm}

\theoremstyle{definition}
\newtheorem{example}[theorem]{Example}
\newtheorem{definition}[theorem]{Definition}
\newtheorem{remark}[theorem]{Remark}

\numberwithin{equation}{section}

\newcommand{\Canakci}{\c{C}anak\c{c}\i}
\newcommand*\circled[1]{\tikz[baseline=(char.base)]{
            \node[shape=circle,draw,inner sep=2pt] (char) {#1};}}

\title[Infinite friezes and triangulations of annuli]{Infinite friezes and triangulations of annuli}

\author[K. Baur]{Karin Baur}
\address{School of Mathematics, University of Leeds, Leeds, United Kingdom. 
On leave from the University of Graz, Graz, Austria} 
\email{k.u.baur@leeds.ac.uk}
\author[\.{I}. \Canakci]{\.{I}lke \Canakci}
\address{Department of Mathematics, VU Amsterdam, 1081 HV Amsterdam, the Netherlands }
\email{i.canakci@vu.nl}
\author[K. M. Jacobsen]{Karin M. Jacobsen}
\address{Fakult\"at f\"ur Mathematik, Universit\"at Bielefeld, 33501 Bielefeld, Germany}
\email{karin.jacobsen@ntnu.no}
\author[M. Kulkarni]{Maitreyee C. Kulkarni}
\address{Max-Planck-Institut f\"ur Mathematik, Bonn, Germany}
\email{mkulkarni@mpim-bonn.mpg.de}
\author[G. Todorov]{Gordana Todorov}
\address{Department of Mathematics, Northeastern University, Boston, United States}
\email{g.todorov@northeastern.edu}


\begin{document}

\begin{abstract}
It is known that any infinite periodic frieze comes from a triangulation of an annulus by \cite[Theorem 4.6]{bpt}. In this 
paper we show that each infinite periodic frieze determines a triangulation of an annulus in essentially a unique way. 
Since each triangulation of an annulus determines a pair of friezes, we study such pairs and 
show how they determine each other. 
We study associated module categories and determine the growth coefficient of the pair of friezes in 
terms of modules as well as their quiddity sequences. 
\end{abstract}

\maketitle

\tableofcontents

%
\section{Introduction}  


The notion of {\em finite friezes} was defined in 1971 by Coxeter in \cite{cox}. 
A finite frieze is a grid consisting of a finite number of rows of positive integers satisfying a determinant rule for each diamond 
in the grid. Each row is of infinite length. The first two and the last two rows are fixed, with entries 0 and 1 as shown in Figure~\ref{fig:finitefrieze}. Because of the determinant rule, the first non-trivial row completely determines the frieze and is called the \emph{quiddity row}. Finite friezes and their properties are well understood. It was proven by Conway and Coxeter (\cite{CC1,CC2}) that every finite frieze has periodic quiddity row and they showed that: 
\begin{thm} (\cite[Question 27]{CC1,CC2})
There is a bijection between periodic finite friezes and triangulations of regular polygons.
\end{thm}

\begin{figure}[ht!]
\begin{center}
\begin{tikzpicture}[scale=.7, inner sep=4pt]
\node (00) at (-6,6) {$\ldots$};
\node (01) at (-4,6) {0};
\node (02) at (-2,6) {0};
\node (ai+) at (0,6) {0};
\node (03) at (2,6) {0};
\node (04) at (4,6) {0};
\node (05) at (6,6) {$\ldots$};

\node (10) at (-7,5) {$\ldots$};
\node (11) at (-5,5) {1};
\node (12) at (-3,5) {1};
\node (12) at (-1,5) {1};
\node (12) at (1,5) {1};
\node (14) at (3,5) {1};
\node (15) at (5,5) {1};
\node (16) at (7,5) {$\ldots$};

\node (ai+) at (-6,4) {$\ldots$};
\node (ai) at (-4,4) {1};
\node (ai+) at (-2,4) {3};
\node (ai+) at (0,4) {1};
\node (aj-) at (2,4) {3};
\node (aj) at (4,4) {1};
\node (ai+) at (6,4) {$\ldots$};

\node (10) at (-7,3) {$\ldots$};
\node (11) at (-5,3) {2};
\node (12) at (-3,3) {2};
\node (12) at (-1,3) {2};
\node (12) at (1,3) {2};
\node (14) at (3,3) {2};
\node (15) at (5,3) {2};
\node (16) at (7,3) {$\ldots$};

\node (ai+) at (-6,2) {$\ldots$};
\node (ai) at (-4,2) {3};
\node (ai+) at (-2,2) {1};
\node (ai+) at (0,2){3};
\node (aj-) at (2,2) {1};
\node (aj) at (4,2) {3};
\node (ai+) at (6,2) {$\ldots$};

\node (10b) at (-7,1) {$\ldots$};
\node (11b) at (-5,1) {1};
\node (12b) at (-3,1) {1};
\node (12b) at (-1,1) {1};
\node (13b) at (1,1) {1};
\node (14b) at (3,1) {1};
\node (15b) at (5,1) {1};
\node (16b) at (7,1) {$\ldots$};

\node (00b) at (-6,0) {$\ldots$};
\node (01b) at (-4,0) {0};
\node (02b) at (-2,0) {0};
\node (ai+b) at (0,0) {0};
\node (03b) at (2,0) {0};
\node (04b) at (4,0) {0};
\node (05b) at (6,0) {$\ldots$};

\end{tikzpicture}
\end{center}
\caption{A finite frieze.}
\label{fig:finitefrieze}
\end{figure}

\noindent Caldero and Chapoton showed that finite friezes are connected to the theory of cluster algebras via the Caldero--Chapoton map and through the association with triangulations of regular polygons \cite{cc}. 

In this paper, we consider another type of friezes called {\em infinite periodic friezes} (see Figure~\ref{infinitefrieze} and Definition \ref{periodic}). Similar to finite friezes, infinite periodic friezes start with a row of 0's and a row of 1's; their quiddity row is also periodic. But the number of rows in this case is infinite. In 2017, it was shown in \cite{bpt} (also see Def~\ref{def:triang-frieze}) that:
\begin{thm} (\cite[Theorem 4.6]{bpt}) Each infinite periodic frieze corresponds to a triangulation of an annulus.
\end{thm}
We will see later that each triangulation of an annulus defines a pair of infinite periodic friezes; one 
associated with each boundary component. 
One of the goals of this paper is to establish a one-to-one correspondence between such pairs of  infinite periodic friezes and triangulations of an annulus, see Proposition~\ref{prop:frieze-tria}. 
In order to get such a correspondence, we define skeletal triangulations and skeletal friezes. A \emph{skeletal triangulation} is a triangulation with only bridging arcs in it (arcs connecting the two boundary components).
A \emph{skeletal frieze} is a frieze whose quiddity sequence has no 1's in it 
and which is different from $q=(2,\dots,2)$. 
With this in mind, we prove the following theorem.
\begin{customthm}{A}\label{thmA}
(Theorem~\ref{thm:unique-frieze}) 
Given a skeletal frieze, there is exactly one other skeletal frieze such that this pair of friezes corresponds to triangulation of an annulus.
\end{customthm}

Every skeletal frieze (or skeletal triangulation
of an annulus) represents an infinite family of infinite periodic friezes (or triangulations of an annulus) (see Theorem~\ref{thm:unique-frieze}). Every skeletal frieze corresponds to a non-oriented cyclic quiver and each such quiver will recover the corresponding pair of skeletal friezes (Theorem~ \ref{thm:quiver to quiddity}).

An interesting phenomenon of infinite periodic friezes is that for a frieze of period $n$, the difference between an entry in its $n$th row and the entry right above it in row $n-2$ is constant. This difference is called the \emph{growth coefficient} of the frieze. It is shown in \cite[Theorem 3.4]{bfpt} that for a pair of friezes coming from a triangulation of an annulus, the growth coefficients of the two friezes are equal. 
We show that:
\begin{customthm}{B}\label{thmB}(Corollary~\ref{cor:frieze-formula}) Consider an infinite periodic frieze with quiddity sequence \(q= (a_1, a_2, \ldots, a_n)\). Then its growth coefficient is 
\[s_q=\left(\sum_{I}(-1)^{\ell_I} \prod_{k\in I}a_k\right)+\delta_n\] where $I$'s are certain subsets of \{$1,2, \ldots, n$\}. The integers $\ell_I$ and $\delta_n$ are defined as in Section \ref{sec:formula-matrices}.
\end{customthm}

Infinite periodic friezes are also related to cluster algebras through the corresponding triangulations. This gives us a relation to cluster categories of type \(\widetilde{A}\) as defined by \cite{bmrrt}, through the Caldero--Chapoton map \cite{cc}. In particular, we explain in Sections \ref{sec:frieze-to-category} and \ref{sec:modules-growth} how infinite periodic friezes are related to the non-homogeneous tubes in the Auslander--Reiten quivers of cluster-tilted algebras of type \(\widetilde A\). Hence we get a representation theoretic interpretation of  infinite periodic frieze patterns and their growth coefficients.

\begin{customthm}{C}\label{thmC} (Corollary~\ref{cor:growth-cc-map})
Let $q=(a_1,\dots, a_n)$ be a quiddity sequence, let $\cB$ be the associated rank $n$ tube. Let 
$M$ be any indecomposable in $\cB$ at level $n$ and 
$\widetilde{M}$ be the indecomposable right below it in the Auslander--Reiten quiver, at level $n+2$. 
Then we have 
\[
s(M)-s(\widetilde{M})=s_{q},
\]
where $s(M)$ is a generalized version of the number of submodules of $M$ as in Section \ref{sec:frieze-to-category}.
\end{customthm}
The paper is organized as follows.
In Section~\ref{sec:notions}, we recall the notion of friezes and triangulations of annuli. 
Moreover, we introduce skeletal friezes and skeletal triangulations and show how these are related to each other. 
In Section~\ref{sec:pairs}, we describe how to associate a pair of  infinite periodic friezes to a triangulation, and study the uniqueness of this correspondence when restricted to skeleta. Furthermore, we give a direct map between non-oriented cyclic quivers and skeletal friezes by associating a pair of quiddity sequences to every non-oriented cycle, see Section~\ref{sec:quiv-quid}. 
This association has indirectly been known before, via triangulations. To our knowledge, 
the direct map is new. 
In Section~\ref{sec:growth} we recall the growth coefficient of an infinite periodic frieze and some results 
from~\cite{bfpt}. We then state our result (Corollary~\ref{cor:frieze-formula}) which computes the growth coefficient 
using only the quiddity sequence. This leads us to the representation theoretic interpretation of the differences of entries in the 
 infinite periodic friezes in terms of certain modules, see Theorem~\ref{thm:repth}, and hence of the growth 
coefficient, see Corollary~\ref{cor:growth-cc-map}.

%
\section{Triangulations, friezes and quivers} \label{sec:notions}

We have three players here: triangulations, friezes and quivers. 
There are clear correspondences between them, some of which are already well known 
and some which we describe in the paper and prove their properties. In order to have uniqueness in such 
correspondences (up to certain factors) we consider 
certain reduced versions for each of these notions, 
namely when they are ``skeletal", 
so that we can go easily between the three players. 

\subsection{Finite and infinite friezes}

Friezes were introduced by Coxeter in 1971 \cite{cox}. 
A {\em frieze} consists of a possibly infinite number of rows of 
positive integers as shown in Figure \ref{infinitefrieze}. The first and second rows consist of 0's 
and 1's respectively. The rows of 0's and 1's are called {\em trivial}. Every diamond of adjacent entries
\begin{center}
 $b$ \\
$a$ \qquad $d$ \\ 
$c$
\end{center}

\noindent in a frieze satisfies the determinant rule $ad-bc=1$.
The {\em order} of a frieze is the number of 
non-trivial rows in the frieze. A frieze is called {\em finite} if it has finite order. The last row of a finite 
frieze consists of 0's and the second to last of 1's.  A frieze is {\it infinite} if it has infinite order. An example of a finite frieze is shown in Figure \ref{fig:finitefrieze}. That frieze has order $3$.
\begin{definition} \label{periodic}A frieze $\cF$ is determined 
by its first non-trivial row $(a_i)_{i\in \Z}$, called the \emph{quiddity row}. 
A frieze is called {\em periodic} if the quiddity row is periodic. So for any $i$, an $n$-periodic frieze $\cF$ is 
determined completely by an $n$-tuple $q\!=\!(a_i,\!..,a_{i+n-1}\!)$ in its quiddity 
row. 
Every such tuple is a {\em quiddity sequence of $\cF$}. 
\end{definition}
Throughout the paper we will consider infinite periodic friezes, which we will sometimes simply call \emph{infinite friezes}. 
We always consider quiddity sequences up to cyclic equivalence, i.e. 
$(a_1,\dots, a_n)\sim (a_2,\dots, a_n,a_1\!)$. 

\begin{figure}[ht!]
\begin{center}
\begin{tikzpicture}[scale=.8, inner sep=4pt]
\node (00) at (-6,6) {$\ldots$};
\node (01) at (-4,6) {0};
\node (02) at (-2,6) {0};
\node (ai+) at (0,6) {0};
\node (03) at (2,6) {0};
\node (04) at (4,6) {0};
\node (05) at (6,6) {$\ldots$};
\node (10) at (-7,5) {$\ldots$};
\node (11) at (-5,5) {1};
\node (12) at (-3,5) {1};
\node (12) at (-1,5) {1};
\node (12) at (1,5) {1};
\node (14) at (3,5) {1};
\node (15) at (5,5) {1};
\node (16) at (7,5) {$\ldots$};

\node (a-2) at (-6,4) {$\ldots$};
\node (a-1) at (-4,4) {\(a_{-1}\)};
\node (a0) at (-2,4) {\(a_{0}\)};
\node (a1) at (0,4) {\(a_{1}\)};
\node (a2) at (2,4) {\(a_{2}\)};
\node (a3) at (4,4) {\(a_{3}\)};
\node (a4) at (6,4) {$\ldots$};

\node (10) at (-7,3) {$\ldots$};
\node (11) at (-5,3) {\(a_{-2,-1}\)};
\node (12) at (-3,3) {\(a_{-1,0}\)};
\node (12) at (-1,3) {\(a_{0,1}\)};
\node (12) at (1,3) {\(a_{1,2}\)};
\node (14) at (3,3) {\(a_{2,3}\)};
\node (15) at (5,3) {\(a_{3,4}\)};
\node (16) at (7,3) {$\ldots$};

\node (ai+) at (-6,2) {$\ldots$};
\node (ai) at (-4,2) {\(a_{-2,0}\)};
\node (ai+) at (-2,2) {\(a_{-1,1}\)};
\node (ai+) at (0,2) {\(a_{0,2}\)};
\node (aj-) at (2,2) {\(a_{1,3}\)};
\node (aj) at (4,2) {\(a_{2,4}\)};
\node (ai+) at (6,2) {$\ldots$};

\node (ai) at (-5,1) {$\vdots$};
\node (ai+) at (-3,1) {$\vdots$};
\node (ai+) at (-1,1) {$\vdots$};
\node (aj-) at (1,1) {$\vdots$};
\node (aj) at (3,1) {$\vdots$};
\node (aj) at (5,1) {$\vdots$};

\end{tikzpicture}
\end{center}
\caption{The layout of an infinite frieze.} 
\label{infinitefrieze}
\end{figure}
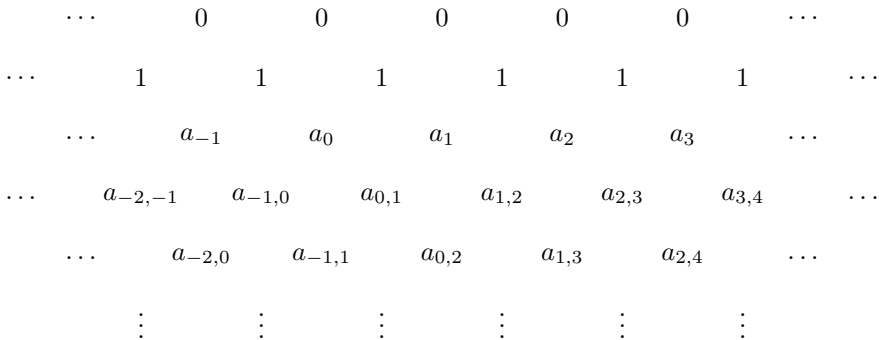

%
\subsection{Skeletal friezes}\label{sec:skeletal_friezes} Quiddity sequences and hence friezes can be simplified through the process of \emph{reduction}, 
as introduced in \cite{CC1, CC2} and described below. This will be a key ingredient for the results in Section~\ref{sec:pairs}: 
These reduced friezes 
help us to establish a bijective 
correspondence between infinite periodic friezes and triangulations of annuli. 

\begin{definition}\label{frieze reduction}
A {\em reduction (at a $1$)} of a quiddity sequence $q$ of a frieze is obtained by 
subtracting 1 from both neighbouring entries in the quiddity sequence 
$q$ provided that the neighbouring entries are $>1$.\\
More precisely, a {\em reduction $\rho$} 
of a quiddity sequence $q=(a_1, \dots,a_r)$ 
at $a_i=1$
is obtained by:\\
(a) reducing both $a_{i-1}$ and $a_{i+1}$ by $1$ and deleting $a_i=1$ in $q$ if $r\geq 3$ and $a_{i-1}, a_{i+1}\geq2$, or \\
(b) if $q=(1,k)$ then deleting $1$ and replacing $k$ by $k$-$2$, i.e. $\rho(q)=(k$-$2)$ if $k\geq 3$.
\end{definition}
We note that reduction primarily makes sense when considering quiddity sequences where no two consecutive entries are equal to one. The full set of cases is discussed below.

\begin{example} 
For instance the reduction of the 
quiddity sequence $q=(1,5)$ is $q'=(3)$,
i.e. $\rho((1,5))=(3)$.
\end{example}

\begin{remark}\label{rmk:reverse reduction} Given a quiddity sequence, we can also perform a 
reverse reduction by inserting a new entry 1 into the sequence and increasing the neighbouring entries by 1. A reverse reduction of the quiddity sequence $q=(2)$ is then $q=(1,4)$, and a reverse reduction of $q=(2,3)$ is $q=(3,1,4)$.
\end{remark}

\begin{definition}
A {\em peripheral triangle} or {\em ear} of a triangulation is a triangle whose sides are: two boundary segments and a peripheral arc (Definition \ref{bridging arc}) (so, all 3 vertices are on the same boundary component). 
\end{definition}

\begin{remark}\label{rem:reduce-triangles}
(a) 
It is well known that finite (infinite) friezes correspond to triangulations of polygons (annuli). We will describe this correspondence for  infinite periodic friezes in Section~\ref{sec:tria-frieze}.
\\
(b) 
In case of both finite and infinite friezes, reduction of a quiddity sequence at a 1 corresponds to deletion of the corresponding peripheral triangle (ear) in the corresponding triangulation (see questions 23, 26, 27 of \cite{CC1,CC2} for finite friezes).
\end{remark}

\begin{remark}\begin{enumerate}
\item If we repeatedly apply reduction to the quiddity sequence of a {\em finite} frieze, we will eventually obtain 
a quiddity sequence where all entries are 1. 
This can be proved using: i) all finite friezes correspond 
to triangulations of polygons, ii) each reduction of the quiddity sequence reduces the number of 
vertices of the polygon 
by Remark~\ref{rem:reduce-triangles}
and iii) the quiddity sequence for the triangle is $(1,1,1)$.
Also a short reference is \cite[Question 23]{CC1,CC2}.
\item 
However, if we repeatedly apply reduction to the quiddity sequence of an {\em infinite} frieze, 
all $1$'s will eventually disappear since otherwise, repeated reductions would give the quiddity sequence $(1)$. 
If the period is 1, the 
quiddity sequence is $(a)$ with $a\ge 2$. For period 2, we can have $(1,a)$ with $a\ge 4$, which we can reduce 
to $(2)$ and $(a,b)$ with $a,b\ge 2$. If the period is at least 3 and if there is an entry 
$1$ in the quiddity sequence, we reduce and get a 
sequence with one entry less, it is still a quiddity sequence of an  infinite periodic frieze (\cite[Theorem 2.7]{tsch}). 
We iterate until all entries are $\ge 2$. Observe that there can never be two entries $a_i=a_{i+1}=1$ 
as by the diamond rule, the entry below them would be 0, so the frieze would not be infinite.
\item
If we start with a quiddity sequence of an  infinite periodic frieze, the final quiddity sequence where all 
entries are $\neq 1$ is independent of the sequence of reductions. 
Indeed, since any reduction on the quiddity sequence corresponds to removing a peripheral triangle (ear) in the surface by Remark \ref{rem:reduce-triangles}(b), and two peripheral triangles may share at most a common point, the order 
in which
we remove peripheral triangles does not matter. 
As a consequence the notions reduced quiddity sequence and reduced frieze 
(Definitions~\ref{def:skeleton-quid} and~\ref{def:skeleton-frieze}) are well-defined. 
\end{enumerate}
\end{remark}

\begin{definition} \label{def:skeleton-quid} 
Given a quiddity sequence $q$ of an infinite periodic frieze, the \emph{reduced 
quiddity sequence $q^s$}
is obtained from $q$ by successively applying reduction to $q$ until there are no occurrences of 1's left.
\end{definition}


\begin{definition}\label{def:skeleton-frieze}
Let $\mathcal F$ be an infinite periodic frieze with quiddity sequence $q$. Then we define 
{\em the reduced frieze $\cF^s$ of} $\cF$ to be the infinite periodic frieze 
of the quiddity sequence $q^s$. 
\end{definition}

\begin{definition}
If $q$ is a quiddity sequence with $q=q^s$ {\bf and if} $q$ is different from $(2,2,\dots, 2)$, 
we will call $q$ a {\em skeletal} quiddity sequence. 
A {\em skeletal frieze} is a frieze obtained from a skeletal quiddity sequence.
\end{definition}

\begin{example} 
If $q=(4,1,2,5)$ then $q^s=(2,4)$. The frieze and its skeletal frieze are shown in Figures \ref{fig:frieze4125} and \ref{fig:frieze24} respectively. 

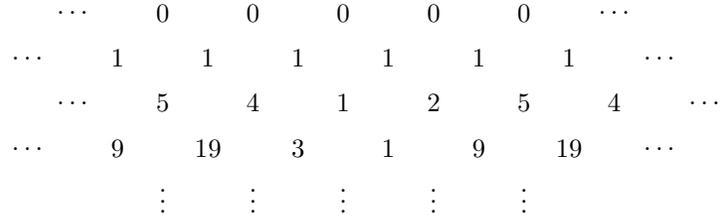
\begin{figure}[ht!]
\begin{tikzpicture}[scale=.6, inner sep=4pt]
\node (00) at (-6,6) {$\ldots$};
\node (01) at (-4,6) {0};
\node (02) at (-2,6) {0};
\node (ai+) at (0,6) {0};
\node (03) at (2,6) {0};
\node (04) at (4,6) {0};
\node (05) at (6,6) {$\ldots$};
\node (10) at (-7,5) {$\ldots$};
\node (11) at (-5,5) {1};
\node (12) at (-3,5) {1};
\node (12) at (-1,5) {1};
\node (12) at (1,5) {1};
\node (14) at (3,5) {1};
\node (15) at (5,5) {1};
\node (16) at (7,5) {$\ldots$};

\node (ai+) at (-6,4) {$\ldots$};
\node (ai) at (-4,4) {5};
\node (ai+) at (-2,4) {4};
\node (ai+) at (0,4) {1};
\node (aj-) at (2,4) {2};
\node (aj) at (4,4) {5};
\node (ai+) at (6,4) {4};
\node (ai+) at (8,4) {$\ldots$};

\node (10) at (-7,3) {$\ldots$};
\node (11) at (-5,3) {9};
\node (12) at (-3,3) {19};
\node (12) at (-1,3) {3};
\node (12) at (1,3) {1};
\node (14) at (3,3) {9};
\node (15) at (5,3) {19};
\node (16) at (7,3) {$\ldots$};

\node (ai) at (-4,2) {$\vdots$};
\node (ai+) at (-2,2) {$\vdots$};
\node (ai+) at (0,2) {$\vdots$};
\node (aj-) at (2,2) {$\vdots$};
\node (aj) at (4,2) {$\vdots$};

\end{tikzpicture}
\caption{Frieze $\cF$ for the quiddity sequence $q=(4,1,2,5)$.}
\label{fig:frieze4125}
\end{figure}

\begin{figure}[ht!]
\begin{tikzpicture}[scale=.6, inner sep=4pt]
\node (00) at (-6,6) {$\ldots$};
\node (01) at (-4,6) {0};
\node (02) at (-2,6) {0};
\node (ai+) at (0,6) {0};
\node (03) at (2,6) {0};
\node (04) at (4,6) {0};
\node (05) at (6,6) {$\ldots$};
\node (10) at (-7,5) {$\ldots$};
\node (11) at (-5,5) {1};
\node (12) at (-3,5) {1};
\node (12) at (-1,5) {1};
\node (12) at (1,5) {1};
\node (14) at (3,5) {1};
\node (15) at (5,5) {1};
\node (16) at (7,5) {$\ldots$};

\node (ai+) at (-6,4) {$\ldots$};
\node (ai) at (-4,4) {2};
\node (ai+) at (-2,4) {4};
\node (ai+) at (0,4) {2};
\node (aj-) at (2,4) {4};
\node (aj) at (4,4) {2};
\node (ai+) at (6,4) {$\ldots$};

\node (10) at (-7,3) {$\ldots$};
\node (11) at (-5,3) {7};
\node (12) at (-3,3) {7};
\node (12) at (-1,3) {7};
\node (12) at (1,3) {7};
\node (14) at (3,3) {7};
\node (15) at (5,3) {7};
\node (16) at (7,3) {$\ldots$};

\node (ai) at (-4,2) {$\vdots$};
\node (ai+) at (-2,2) {$\vdots$};
\node (ai+) at (0,2) {$\vdots$};
\node (aj-) at (2,2) {$\vdots$};
\node (aj) at (4,2) {$\vdots$};

\end{tikzpicture}
\caption{The skeletal frieze $\cF^s$ for the quiddity sequence $q^s= (2,4)$.}
\label{fig:frieze24}
\end{figure}
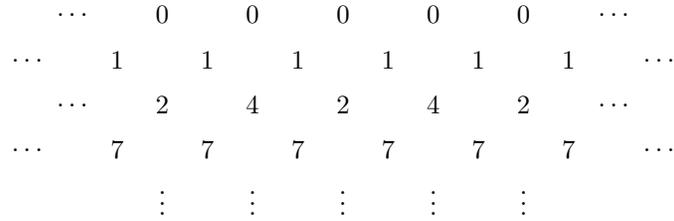
\end{example}

We end this subsection with the special frieze associated to the so called ``trivial quiddity sequence" $q=(2,2,\dots,2)$. This is an infinite periodic frieze which arises from a triangulation of a punctured disc 
as well as from a triangulation of an annulus using spiraling arcs. See Remark \ref{rem:arithmetic}. 
We will later exclude these types of friezes as they form a class of their own which is not 
relevant for our work, and have been well studied elsewhere.

\begin{example}\label{ex:all-2}
Consider the quiddity sequence $q=(2,2,\dots,2)$. It defines an infinite periodic frieze where 
the entries in the $k$-th non-trivial row are all equal to $k+1$. 
See Figure~\ref{fig:frieze-all-2}. 
The entries in each diagonal row grow linearly, this is an example of an {\em arithmetic infinite 
frieze} as studied by Tschabold~\cite{tsch}. We will call this the {\em trivial quiddity sequence}.

\begin{figure}[ht!]
\begin{tikzpicture}[scale=.5, inner sep=4pt]
\node (00) at (-6,6) {$\ldots$};
\node (01) at (-4,6) {0};
\node (02) at (-2,6) {0};
\node (ai+) at (0,6) {0};
\node (03) at (2,6) {0};
\node (04) at (4,6) {0};
\node (05) at (6,6) {$\ldots$};
\node (10) at (-7,5) {$\ldots$};
\node (11) at (-5,5) {1};
\node (12) at (-3,5) {1};
\node (12) at (-1,5) {1};
\node (12) at (1,5) {1};
\node (14) at (3,5) {1};
\node (15) at (5,5) {1};
\node (16) at (7,5) {$\ldots$};

\node (ai+) at (-6,4) {$\ldots$};
\node (ai) at (-4,4) {2};
\node (ai+) at (-2,4) {2};
\node (ai+) at (0,4) {2};
\node (aj-) at (2,4) {2};
\node (aj) at (4,4) {2};
\node (ai+) at (6,4) {$\ldots$};

\node (10) at (-7,3) {$\ldots$};
\node (11) at (-5,3) {3};
\node (12) at (-3,3) {3};
\node (12) at (-1,3) {3};
\node (12) at (1,3) {3};
\node (14) at (3,3) {3};
\node (15) at (5,3) {3};
\node (16) at (7,3) {$\ldots$};

\node (ai) at (-4,2) {$\vdots$};
\node (ai+) at (-2,2) {$\vdots$};
\node (ai+) at (0,2) {$\vdots$};
\node (aj-) at (2,2) {$\vdots$};
\node (aj) at (4,2) {$\vdots$};

\end{tikzpicture}
\caption{The infinite periodic frieze with quiddity sequence $q=(2,2,\dots,2)$.} 
\label{fig:frieze-all-2}
\end{figure}
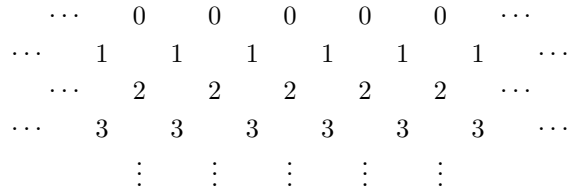
\end{example}

%
\subsection{Triangulations of surfaces and their associated quivers}

Let $S$ be a connected oriented surface with boundary. We denote by $M$, a finite set of 
marked points on the boundary. The pair $(S, M)$ is called a bordered surface with marked points 
if $S$ is non-empty and each connected component of the boundary of $S$ has at least one 
marked point, see~\cite[Section 2]{fst} for details. 

An {\em arc} $\gamma$ in $(S,M)$ is a class of curves equivalent up to homotopy in $S$ which start 
and end at the marked points in $M$. We require that there exists a curve $\gamma$ in $(S,M)$ for which 
the following 
conditions hold (we use the same notation for the arc and its representative): 
\begin{itemize}
\item $\gamma$ has no self-intersections,
\item $\gamma$ does not intersect the set $M$ and the boundary of the surface $S$, except at its endpoints,
\item $\gamma$ is not contractible to a subset of the boundary of $S$. 
\end{itemize}

Two arcs are compatible if they have representatives which do not intersect each other inside $S$. A maximal collection of such pairwise compatible arcs is called a {\em triangulation} of $(S,M)$. A triangulation of $(S,M)$ has two types of arcs. 
\begin{definition} \label{bridging arc} An arc that connects two marked 
points on the same boundary component is called a {\em peripheral arc}. An arc that connects 
two marked points on different boundary components is 
a {\em bridging arc}.
\end{definition}

We recall the definition of the quiver of a triangulation of a surface, as illustrated in 
Figure~\ref{fig:triang-quiver}.

\begin{definition}
Let $\mathcal T$ be a triangulation of a surface $(S,M)$. We associate a {\em quiver} $Q_\mathcal T$ to $\mathcal T$ as: 
\begin{description}
\item[Vertices] For every arc in $\mathcal T$ we have a vertex in $Q_\mathcal T$,
\item[Arrows] Suppose that \(i\) and \(j\) are arcs in the triangulation sharing one endpoint such that $i$ is a direct predescessor to $j$ with respect to anti-clockwise rotation at their shared endpoint. Then there is an arrow from $i$ to $j$ 
in $Q_\mathcal T$.
\end{description}
\end{definition}

\begin{figure}
\begin{tikzpicture}[scale=0.36]

    \draw[thick,fill=background] (5,0) arc (0:360:5); 
     \draw[thick,fill=white]  (1.5,0) arc (0:360:1.5); 


    \node[shape=circle,fill=black,scale=.65] (o1) at (30:5) {}; 
    \node[shape=circle,fill=black,scale=.65] (o2) at (270:5) {};
    \node[shape=circle,fill=black,scale=.65] (o3) at (150:5) {};
    \node[shape=circle,fill=black,scale=.65] (i1) at (90:1.5) {};
    \node[shape=circle,fill=black,scale=.65] (i2) at (270:1.5) {};

    \node[shape=circle,fill=arc1,scale=.65] (q1) at (41:3) {};  
    \node[shape=circle,fill=arc1,scale=.65] (q2) at (342:3) {};
    \node[shape=circle,fill=arc1,scale=.65] (q3) at (270:3) {};
    \node[shape=circle,fill=arc1,scale=.65] (q4) at (198:3) {};
    \node[shape=circle,fill=arc1,scale=.65] (q5) at (139:3) {};

   \draw[thick] (o2) to (i2); 
   \draw[thick][out=270,in=340] (o1) to (i2);
   \draw[thick][out=270,in=200] (o3) to (i2);
   \draw[thick] (o1) to (i1);
   \draw[thick] (o3) to (i1);
   
   \draw[thick, ->, arc1, shorten >= 2pt, shorten <= 2pt] (q1) to (q2); 
   \draw[thick, <-, arc1, shorten >= 2pt, shorten <= 2pt][out=270,in=340] (q2) to (q3);
   \draw[thick, <-, arc1, shorten >= 2pt, shorten <= 2pt][out=200,in=270] (q3) to (q4);
   \draw[thick, ->, arc1, shorten >= 2pt, shorten <= 2pt] (q4) to (q5);
   \draw[thick, <-, arc1, shorten >= 2pt, shorten <= 2pt] [out=45,in=145] (q5) to (q1);

\end{tikzpicture}
\qquad
\begin{tikzpicture}[scale=0.36]

    \draw[thick,fill=background] (5,0) arc (0:360:5); 
     \draw[thick,fill=white]  (1.5,0) arc (0:360:1.5); 
    
    \node[shape=circle,fill=black,scale=.65] (o1) at (30:5) {}; 
    \node[shape=circle,fill=black,scale=.65] (o2) at (270:5) {};
    \node[shape=circle,fill=black,scale=.65] (o3) at (150:5) {};
    \node[shape=circle,fill=black,scale=.65] (i1) at (90:1.5) {};
    \node[shape=circle,fill=black,scale=.65] (i2) at (270:1.5) {};

    \node[shape=circle,fill=arc1,scale=.65] (q1) at (41:3) {};  
    \node[shape=circle,fill=arc1,scale=.65] (q2) at (342:3) {};
    \node[shape=circle,fill=arc1,scale=.65] (q3) at (270:4) {};
    \node[shape=circle,fill=arc1,scale=.65] (q4) at (198:3) {};
    \node[shape=circle,fill=arc1,scale=.65] (q5) at (139:3) {};

 \draw [thick] (o1) to [out=280,in=0] (0,-4)
        to [out=180,in=260] (o3) ; 
   \draw[thick][out=270,in=340] (o1) to (i2);
   \draw[thick][out=270,in=200] (o3) to (i2);
   \draw[thick] (o1) to (i1);
   \draw[thick] (o3) to (i1);
   
   \draw[thick, ->, arc1, shorten >= 2pt, shorten <= 2pt] (q1) to (q2); 
   \draw[thick, ->, arc1, shorten >= 2pt, shorten <= 2pt][out=270,in=20] (q2) to (q3);
   \draw[thick, ->, arc1, shorten >= 2pt, shorten <= 2pt][out=160,in=270] (q3) to (q4);
   \draw[thick, ->, arc1, shorten >= 2pt, shorten <= 2pt] (q4) to (q5);
   \draw[thick, <-, arc1, shorten >= 2pt, shorten <= 2pt] [out=45,in=145] (q5) to (q1);
   \draw[thick, ->, arc1, shorten >= 2pt, shorten <= 2pt] [out=300,in=240](q4) to (q2);

\end{tikzpicture}
\caption{Examples of quivers $Q_\mathcal T$ and $Q_{\mathcal T'}$ from two triangulations of $C_{3,2}$.}
\label{fig:triang-quiver}
\end{figure}
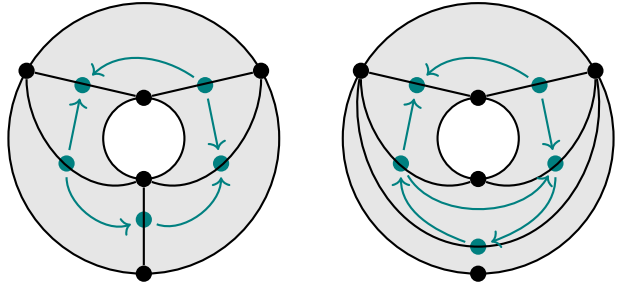

\begin{remark}
We observe that (isotopy classes of) infinite curves which start at one marked point and spiral 
around infinitely many times have also been used in triangulations of surfaces, as introduced in~\cite[Section 3.2]{bbm}. 
Such arcs are called {\em spiraling arcs}. See Figure~\ref{fig:all-twos} for examples. 
However such arcs do not appear in the results of this paper.
\end{remark}


\subsection{Triangulations and friezes}\label{sec:tria-friezes} 

Now let us recall the correspondence between triangulations of regular polygons and finite friezes \cite[Section 3]{CC1,CC2}. 
First, note that every finite frieze of order $n$ has period $(n+3)$ \cite[Question 26]{CC1,CC2}. In fact, every finite frieze of order $n$ comes from a triangulation of a regular $(n+3)$-gon \cite[Question 28 and 29]{CC1,CC2}. The correspondence can be described as follows: let $\cF$ be a frieze of period $(n+3)$ with quiddity sequence $q$. Then the $(n+3)$ entries of $q$ count the number of triangles at the vertices in a triangulation of an $(n+3)$-gon. 

Every triangulation of an annulus gives rise to an infinite periodic frieze and every infinite periodic 
frieze can be realized by a triangulation of an annulus (\cite[Theorem 3.7, Theorem 4.6]{bpt}). For this reason, the surface $S$ we consider will always be an annulus. We denote by $C_{m,n}$ an annulus with $m$ marked points on one boundary component $B_1$ and $n$ marked points on the other boundary component $B_2$. We will refer to them as inner and outer boundary components of annulus if we use a geometric presentation of the annulus in figures. We will label the marked points on $B_1$ anti-clockwise around the outer boundary and the marked points on $B_2$ anti-clockwise around the inner boundary.

\begin{remark}\label{rem:arithmetic}
It is worth pointing out that there are infinite periodic friezes which can also be given 
by triangulations of a punctured disk. These are the arithmetic infinite friezes 
from Examples~\ref{ex:all-2} and~\ref{ex:spiraling}. 
\end{remark}



%
\subsection{Triangulations of annuli and friezes}\label{sec:tria-frieze}

We recall from \cite[Theorem 4.6]{bpt} that for every infinite periodic frieze 
we can construct a triangulation of an annulus such that the number of triangles at the marked 
points on one boundary give the quiddity sequence of the frieze (as in 
Definition~\ref{def:triang-frieze} below). 
While finite friezes are in bijection with triangulations of polygons, infinite periodic friezes arise 
from triangulations of annuli, but not in a unique way. 
In fact, each infinite periodic frieze gives rise to a family of triangulations of an annulus. If we restrict to 
skeletal friezes, the triangulation is essentially unique as we will see in Proposition~\ref{prop:frieze-tria}. 

As a consequence we show in Theorem~\ref{thm:unique-frieze} that if $\mathcal F_1$ is skeletal, 
it determines another skeletal frieze $\mathcal F_2$ essentially uniquely such that the pair arises from a triangulation of an annulus. We start with a triangulation of an annulus and associate two infinite periodic friezes to it. 

\begin{definition}\label{def:triang-frieze}
Let $\mathcal T$ be a triangulation on an annulus $C_{m,n}$. For each of the two boundaries $B_1$ and $B_2$ define the {\em quiddity sequences} $q_1=(a_1,\ldots,a_m)$ and $q_2=(b_1,\ldots,b_n)$ as the number of triangles at marked points on $B_1$ and $B_2$ respectively. The friezes corresponding to $q_1$ and $q_2$ will be denoted by $\mathcal F_1$ and $\mathcal F_2$ respectively.
\end{definition}

Note that we will not distinguish between the two annuli $C_{m,n}$ and $C_{n,m}$.  Let $\mathcal T$ be a triangulation of $C_{m,n}$ and $\mathcal T^{\prime}$ be the triangulation of $C_{n,m}$ obtained by exchanging the boundaries of $C_{m,n}$, then $Q_{\mathcal T}$ and $Q_{\mathcal T^{\prime}}$ are isomorphic quivers.

\begin{example} Consider the triangulation $\mathcal T$ of an annulus $C_{3,2}$ shown in Figure~\ref{TrngAnnulus}, the corresponding quiddity sequences are $q_1=(2,3,3)$ and $q_2=(3,4)$. The frieze $\cF_1$ corresponding to the boundary $B_1$ 
is defined to be the frieze which
has quiddity sequence $q_1=(2,3,3)$ and therefore is as follows: 
\begin{center}
\noindent $\ldots$ \qquad 0 \qquad 0 \qquad 0 \qquad 0 \qquad 0 \qquad 0 \qquad $\ldots$\\
\hspace{1cm} $\ldots$ \qquad 1 \qquad 1 \qquad 1 \qquad 1 \qquad 1 \qquad 1 \qquad $\ldots$\\
$\ldots$ \qquad 3 \qquad 3 \qquad 2 \qquad 3 \qquad 3 \qquad 2 \qquad $\ldots$\\
\hspace{1cm} $\ldots$ \qquad 8 \qquad 5 \qquad 5 \qquad 8 \qquad 5\qquad 5 \qquad $\ldots$\\
\hspace{.8cm} $\ddots$ \qquad \hspace{.8cm} \qquad $\ddots$\\ 
\end{center}

\noindent The frieze $\cF_2$ corresponding to the boundary $B_2$ with quiddity sequence $q_2=(3,4)$ is as follows: 

\begin{center}
\noindent $\ldots$ \qquad 0 \qquad 0 \qquad 0 \qquad 0 \qquad $\ldots$\\
\hspace{1cm} $\ldots$ \qquad 1 \qquad 1 \qquad 1 \qquad 1 \qquad $\ldots$\\
$\ldots$ \qquad 3 \qquad 4 \qquad 3 \qquad 4 \qquad $\ldots$\\
\hspace{1cm} $\ldots$ \qquad 11 \qquad 11 \qquad 11 \qquad 11 \qquad $\ldots$\\
\hspace{1.2cm} $\ddots$ \qquad \hspace{.4cm} \qquad $\ddots$
\end{center}

\begin{center}

\begin{figure}[ht]
\begin{tikzpicture}[scale=0.35]

    \draw[thick,fill=background] (5,0) arc (0:360:5); 
    \draw[thick,fill=white]  (1.5,0) arc (0:360:1.5); 
     
    \node[shape=circle,fill=black,scale=.65] (o1) at (30:5) {}; 
    \node[shape=circle,fill=black,scale=.65] (o2) at (270:5) {};
    \node[shape=circle,fill=black,scale=.65] (o3) at (150:5) {};
    \node[shape=circle,fill=black,scale=.65] (i1) at (90:1.5) {};
    \node[shape=circle,fill=black,scale=.65] (i2) at (270:1.5) {};
    
    \node at (30:5.8) {$\tiny\circled{3}$}; 
    \node at (270:5.8) {$\tiny\circled{2}$};
    \node at (150:5.8) {$\tiny\circled{3}$};
    \node  at (90:0.8) {$\tiny\circled{3}$};
    \node at (270:0.8) {$\tiny\circled{4}$};
 
    \node at (90:5.5) {$B_1$};
    \node at (0:2) {$B_2$};
     
   \draw[thick] (o2) to (i2); 
   \draw[thick][out=270,in=340] (o1) to (i2);
   \draw[thick][out=270,in=200] (o3) to (i2);
   \draw[thick] (o1) to (i1);
   \draw[thick] (o3) to (i1);
\end{tikzpicture}
\caption{A triangulation of an annulus $C_{3,2}$ with number of triangles at each marked point.}
\label{TrngAnnulus}
\end{figure}
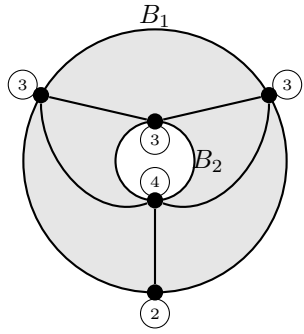
\end{center}
\end{example}

\begin{example}\label{ex:spiraling}
The triangulation of an annulus giving rise to the trivial quiddity sequence $(2,2,2,2)$ is given 
by arcs spiraling around a non-contractible curve in the annulus as shown in Figure~\ref{fig:all-twos}.

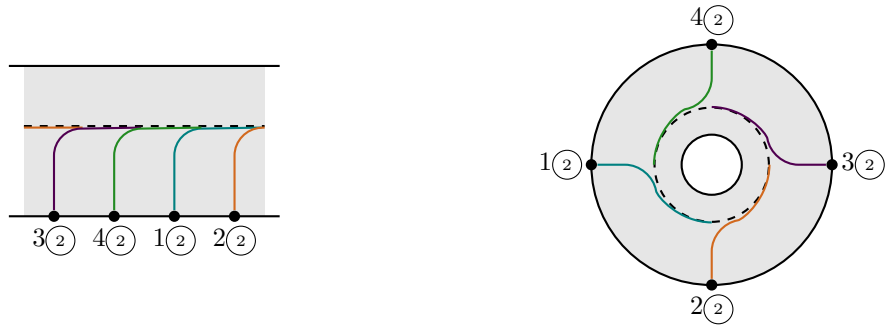
\begin{figure}[ht]
\begin{subfigure}{0.45\linewidth}
\centering
\begin{tikzpicture}[scale=.4]
\fill[background] (0,0) rectangle (8,5) ; 
\draw[thick] (-.5,0) -- (8.5,0); 
\draw[thick] (-.5,5) -- (8.5,5); 
\node[inner sep=1.5pt, fill, circle] (3) at (1, 0) {};
\node[inner sep=1.5pt, fill, circle] (4) at (3, 0) {};
\node[inner sep=1.5pt, fill, circle] (1) at (5, 0) {};
\node[inner sep=1.5pt, fill, circle] (2) at (7, 0) {};
\node[below] (1txt) at (1) {$1\tiny\circled{2}$};
\node[below] (2txt) at (2) {$2\tiny\circled{2}$};
\node[below] (3txt) at (3) {$3\tiny\circled{2}$};
\node[below] (4txt) at (4) {$4\tiny\circled{2}$};
\draw[thick, dashed] (0, 3)--(8, 3);
\begin{scope}[thick, rounded corners=10pt]
\draw[color=arc1] (1) -- ++(0,2.9) -- ++ (2.9,0.05); 
\draw[color=arc2] (2) -- ++(0,2.9) -- ++ (1,0.05); 
\draw[color=arc2] (0,2.95) -- ++(1.95,0); 
\draw[color=arc3] (3) -- ++(0,2.9) -- ++ (2.9,0.05); 
\draw[color=arc4] (4) -- ++(0,2.9) -- ++ (2.9,0.05); 
\end{scope}
\end{tikzpicture}
\end{subfigure}
 \begin{subfigure}{0.45\linewidth}
 \centering
\begin{tikzpicture}[scale=.4]
\draw[thick,fill=background, radius=4cm] (0,0) circle; 
\draw[thick,fill=white!10, radius=1cm] (0,0) circle; 
\draw[thick, radius=1.9cm, dashed] (0,0) circle;

\node[inner sep=1.5pt, fill, circle] (3) at (0:4) {};
\node[inner sep=1.5pt, fill, circle] (4) at (90:4) {};
\node[inner sep=1.5pt, fill, circle] (1) at (180:4) {};
\node[inner sep=1.5pt, fill, circle] (2) at (270:4) {};
\node[left] (1txt) at (1) {$1\tiny\circled{2}$};
\node[below] (2txt) at (2) {$2\tiny\circled{2}$};
\node[right] (3txt) at (3) {$3\tiny\circled{2}$};
\node[above] (4txt) at (4) {$4\tiny\circled{2}$};

\begin{scope}[thick, rounded corners=10pt]
\draw[color=arc1] (1) -- ++(0:2) .. controls (210:2.1) and (240:2.2) .. (270:1.92); 
\draw[color=arc2] (2) -- ++(90:2) .. controls (300:2.1) and (330:2.2) .. (360:1.92); 
\draw[color=arc3] (3) -- ++(180:2) .. controls (30:2.1) and (60:2.2) .. (90:1.92); 
\draw[color=arc4] (4) -- ++(270:2) .. controls (120:2.1) and (150:2.2) .. (180:1.92); 
\end{scope}

\end{tikzpicture}
\end{subfigure}
\caption{The asymptotic triangulation of an annulus corresponding to \(q=(2,2,2,2)\). 
The figure on the left shows a fundamental domain for the surface, 
where 
the two vertical ends of the shaded region are identified. It is often convenient to draw arcs 
in this way.}
\label{fig:all-twos}
\end{figure}
\end{example}

\begin{remark}\label{rem:no-arithm}
We will from now on always exclude asymptotic triangulations, the trivial 
quiddity sequences $(2,2,\dots,2)$ and their infinite periodic friezes (as in Figure~\ref{fig:frieze-all-2}). 
\end{remark}

%
\section{Pairs of friezes}\label{sec:pairs}
%
In this section we characterize the relation between the two infinite periodic friezes associated to a triangulation.
\subsection{Skeletal triangulations} 
In the section \ref{sec:skeletal_friezes}, we studied skeletal friezes.  Similarly we can define skeletal triangulations. To a triangulation $\mathcal T$ of an annulus with marked points, we associate a skeletal triangulation $\mathcal T^s$ of a new annulus with marked points. We show that the essential properties that 
we are concerned about are preserved when going from a triangulation/frieze
to the corresponding skeletal triangulation/frieze. 
Since skeletal triangulations/friezes 
are easier to deal with, we will work with them. 
Recall that we only allow finite arcs in triangulations.

\begin{definition}
A {\em flip} of an arc in a triangulation $\cT$ is the replacement of the arc by the unique other 
arc in the quadrilateral formed by the two triangles incident with the arc. 
\end{definition}
An example of a flip is shown in Figure~\ref{fig:triang-quiver}. 
The vertical arc in the annulus on the left hand is 
flipped to a peripheral arc in the right hand.

\begin{definition}
A peripheral arc in a triangulation $\cT$ is a {\em bounding arc} (for $\cT$) if its flip 
is a bridging arc as defined in Definition \ref{bridging arc}.
\end{definition}

\begin{remark}\label{partition}
Every bounding arc of a triangulation splits the annulus into two parts. 
One part is a triangulated polygon and the other is a triangulated annulus. 
\end{remark}

\begin{lemma}\label{lm:bounding-exist}
Every triangulation with a peripheral arc has a bounding arc. 
\end{lemma}

\begin{proof}
Assume that $\cT$ has a peripheral arc $\gamma$. Let $\gamma'$ be maximal and peripheral 
above $\gamma$ ($\gamma$ lies between $\gamma'$ and the boundary it is attached to and there is no other peripheral arc 
above $\gamma'$). Denote the endpoints of 
$\gamma$ by $P_1$ and $P_2$. Then $\gamma'$ is incident with two 
triangles of $\cT$: one consists only of peripheral arcs, say with vertices $P_1,P_2,P_3$ along the boundary and the other 
one has vertices $P_1,Q,P_2$ with 
$Q$ on the other boundary. Flipping $\gamma$ results in the arc connecting $P_3$ with $Q$, a briding arc. 
Therefore, $\gamma$ is bounding. 
\end{proof}

\begin{definition}\label{s_T}
Let $\cT$ be a triangulation of an annulus. The {\em skeletal triangulation 
$\cT^s$ of $\cT$} is obtained by cutting along all bounding arcs and removing 
the triangulated polygon attached with them.
\end{definition}

\begin{example}\label{ex:skeleton-tr}
In Figure \ref{fig:skeleton-triangulation} we give an example of a triangulation $\mathcal T$ of an annulus and the corresponding skeletal triangulation $\mathcal T^s$.

\begin{figure}[ht!]
 \begin{subfigure}[b]{0.35\linewidth}
 \centering

	\begin{tikzpicture}[scale=0.35]
    \draw[thick,fill=background] (5,0) arc (0:360:5); 
    \draw[thick,fill=white]  (1.5,0) arc (0:360:1.5); 
     
    \node[shape=circle,fill=black,scale=.65] (o1) at (30:5) {};
    \node[shape=circle,fill=black,scale=.65] (o2) at (270:5) {};
    \node[shape=circle,fill=black,scale=.65] (o3) at (150:5) {};
    \node[shape=circle,fill=black,scale=.65] (i1) at (90:1.5) {};
    \node[shape=circle,fill=black,scale=.65] (i2) at (270:1.5) {};
 
   \draw[thick,arc1][out=270,in=340] (o1) to (i2);
   \draw[thick,arc1][out=270,in=200] (o3) to (i2);
   \draw[thick,arc1] (o1) to (i1);
   \draw[thick,arc1] (o3) to (i1);
   \draw [thick,arc2] (o1) to [out=280,in=0] (0,-4)
        to [out=180,in=260] (o3) ;
	\end{tikzpicture}
  \end{subfigure}
  \begin{subfigure}[b]{0.4\linewidth}
  \centering

\begin{tikzpicture}[scale=0.35]

    \draw[thick,fill=background] (5,0) arc (0:360:5); 
    \draw[thick,fill=white]  (1.5,0) arc (0:360:1.5); 
 
    \node[shape=circle,fill=black,scale=.65] (o1) at (30:5) {};
    \node[shape=circle,fill=black,scale=.65] (o3) at (150:5) {};
    \node[shape=circle,fill=black,scale=.65] (i1) at (90:1.5) {};
    \node[shape=circle,fill=black,scale=.65] (i2) at (270:1.5) {};

   \draw[thick,arc1][out=270,in=340] (o1) to (i2);
   \draw[thick,arc1][out=270,in=200] (o3) to (i2);
   \draw[thick,arc1] (o1) to (i1);
   \draw[thick,arc1] (o3) to (i1);


\end{tikzpicture}
\end{subfigure} 
\caption{A triangulation of an annulus $C_{3,2}$ and the corresponding skeletal triangulation of $C_{2,2}$.}
\label{fig:skeleton-triangulation}
\end{figure}
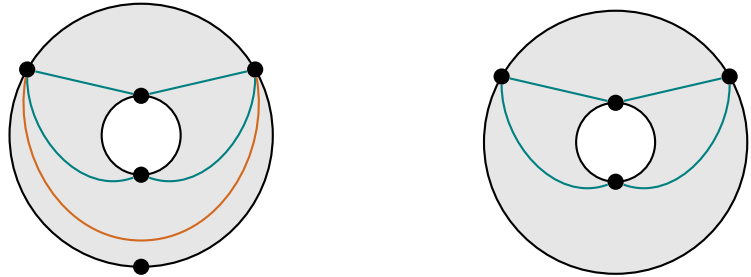
\end{example}

\begin{lemma}\label{lm:skeleton-tri}
Let $\cT$ be a triangulation of an annulus. Then:
\begin{enumerate}
\item The process of constructing the skeletal triangulation $\cT^s$ is well defined.
\item The skeletal triangulation $\cT^s$ has only bridging arcs.
\end{enumerate}
\end{lemma}

\begin{proof} (1) This is clear since bounding arcs do not intersect except possibly at end points, 
and no new bounding arcs are created after cutting along existing bounding arcs. 
Thus the order of removing them does not matter. 

(2) Every non-bounding peripheral arc lies in a polygon given by some bounding arc and the 
corresponding boundary of the annulus between the two endpoints of the bounding arc. In particular, every 
peripheral arc gets removed when going from $\cT$ to $\cT^s$. 
\end{proof}

Note that the skeletal triangulation $\cT^s$ of a triangulation 
$\cT$ of an annulus, is actually a triangulation
of a different annulus (as in the above Figure \ref {fig:skeleton-triangulation}). 
\begin{lemma}
If a triangulation of $C_{m,n}$ has 
$k$ peripheral arcs, then its skeletal triangulation is of an annulus 
$C_{m-k_1, n-k_2}$ where $k_1$ is the number of peripheral arcs on the outer boundary 
and $k_2$ the number of peripheral arcs on the inner boundary (in particular, $k_1 + k_2 =k$). 
\end{lemma}
\begin{proof}
The proof follows since removing one peripheral arc of a triangle with two boundary segments 
corresponds to removing one marked point on that boundary. 
\end{proof}

By Lemma~\ref{lm:skeleton-tri} (2),
all arcs of a skeletal triangulation are bridging. We have the following: 
\begin{lemma}\label{lm:skeletal-property-1} 
Let $\cT$ be a triangulation of an annulus. Then,
$\cT=\cT^s$ if and only if the quiver $Q_{\cT}$ of $\cT$ is a non-oriented cycle. 
\end{lemma}
\begin{proof}
Assume first that $\cT=\cT^s$. Then the only arcs in the annulus are bridging. 
In particular, every triangle of $\cT$ 
has a boundary segment as one of its edges and two bridging arcs. So every triangle gives rise 
to exactly one arrow in $Q_{\cT}$ and every vertex of $Q_{\cT}$ has exactly two arrows incident 
with it. Thus the graph underlying $Q_{\cT}$ is a cycle. 
The arcs in $\cT$ form a collection of fans from each boundary (some of these fans only consist of two arcs). 
Look at the outer boundary. 
Take any two neighboured fans $f_1$ and $f_2$: the right most arc of $f_1$ together with the left most arc of 
$f_2$ form a $V$. Together with the next arc to the left, they form an N, together with the next arc on the right they 
form a \reflectbox{N}. This means that $Q_{\cT}$ has at least one sink and at least one source, respectively. 

Assume now that 
$Q_{\cT}$ is a non-oriented cyclic quiver, and suppose that the corresponding triangulation $\cT$ has 
a peripheral arc $\gamma'$; then $\cT$ also must have at least one bounding arc $\gamma$ 
by Lemma~\ref{lm:bounding-exist}. For $\cT$ 
to be a triangulation, we need $\gamma$ to be part of an internal triangle. 
This internal triangle gives rise to an oriented 3-cycle in $Q_{\cT}$, leading to a contradiction. 
\end{proof}



The following proposition is a consequence of the construction from Section 4 of~\cite{bpt}. 
It uses a construction of a triangulation for a skeletal quiddity sequence. 

\begin{proposition}\label{prop:frieze-tria} Let $q=q^s$ be a skeletal 
quiddity sequence. 
Then up to rotating the inner boundary and its marked points, 
there is a unique skeletal triangulation of an annulus $C_{m,n}$ such that the infinite periodic frieze associated 
to the outer boundary has $q$ as its quiddity sequence. 
\end{proposition} 

\begin{proof}
This follows from the algorithm described in the proof of~\cite[Corollary 4.5]{bpt}: 
Let $q=q^s=(a_1,\dots, a_m)$. Recall that $q\ne (2,2,\dots,2)$. 
Draw an annulus with $m$ vertices on the outer boundary, labeled $1,2,\dots, m$ clockwise. 
For every $a_i$, draw $a_i-1$ arc segments at vertex $i$. 
See the first picture in Figure~\ref{fig:quid-tri} for an illustration.

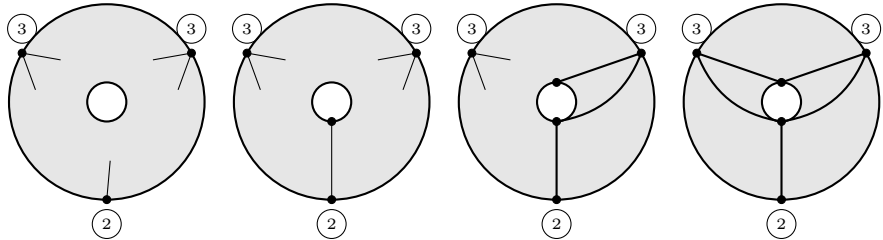
\begin{figure}[ht!]
\begin{tikzpicture}[scale=.26]
\draw[thick,fill=background, radius=5cm] (0,0) circle; 
\draw[thick,fill=white!10, radius=1cm] (0,0) circle; 
\draw[fill=black, radius = .2cm] (30:5) circle node[above ] {$\tiny\circled{3}$};
\draw[fill=black, radius = .2cm] (150:5) circle node[above ] {$\tiny\circled{3}$};
\draw[fill=black, radius = .2cm] (270:5) circle node[below] {$\tiny\circled{2}$};
\draw (30:5) -- +(190:2);
\draw (30:5) -- +(250:2);
\draw (150:5) -- +(-10:2);
\draw (150:5) -- +(290:2);
\draw(270:5) -- +(85:2);

\begin{scope}[xshift =11.5cm]
\draw[thick,fill=background, radius=5cm] (0,0) circle; 
\draw[thick,fill=white!10, radius=1cm] (0,0) circle; 

\draw[fill=black, radius = .2cm] (30:5) circle node[above ] {$\tiny\circled{3}$};
\draw[fill=black, radius = .2cm] (150:5) circle node[above ] {$\tiny\circled{3}$};
\draw[fill=black, radius = .2cm] (270:5) circle node[below] {$\tiny\circled{2}$};

\draw[fill=black, radius = .2cm] (270:1) circle;
\draw (30:5) -- +(190:2);
\draw (30:5) -- +(250:2);
\draw (150:5) -- +(-10:2);
\draw (150:5) -- +(290:2);

\draw (270:5) -- (270:1);
\end{scope}

\begin{scope}[xshift =23cm]
\draw[thick,fill=background, radius=5cm] (0,0) circle; 
\draw[thick,fill=white!10, radius=1cm] (0,0) circle; 

\draw[fill=black, radius = .2cm] (30:5) circle node[above ] {$\tiny\circled{3}$};
\draw[fill=black, radius = .2cm] (150:5) circle node[above ] {$\tiny\circled{3}$};
\draw[fill=black, radius = .2cm] (270:5) circle node[below] {$\tiny\circled{2}$};

\draw[fill=black, radius = .2cm] (270:1) circle;
\draw[fill=black, radius = .2cm] (90:1) circle;
\draw (150:5) -- +(-10:2);
\draw (150:5) -- +(290:2);

\draw[thick] (270:5) -- (270:1);
\draw[thick] (30:5) to[bend left] (270:1);
\draw[thick] (30:5) -- (90:1);
\end{scope}

\begin{scope}[xshift =34.5cm]
\draw[thick,fill=background, radius=5cm] (0,0) circle; 
\draw[thick,fill=white!10, radius=1cm] (0,0) circle; 

\draw[fill=black, radius = .2cm] (30:5) circle node[above ] {$\tiny\circled{3}$};
\draw[fill=black, radius = .2cm] (150:5) circle node[above ] {$\tiny\circled{3}$};
\draw[fill=black, radius = .2cm] (270:5) circle node[below] {$\tiny\circled{2}$};

\draw[fill=black, radius = .2cm] (270:1) circle;
\draw[fill=black, radius = .2cm] (90:1) circle;

\draw[thick] (270:5) -- (270:1);
\draw[thick] (30:5) to[bend left] (270:1);
\draw[thick] (30:5) -- (90:1);
\draw[thick] (150:5) -- (90:1);
\draw[thick] (150:5) to[bend right] (270:1);
\end{scope}

\end{tikzpicture}
\caption{Triangulation from the quiddity sequence $(2,3,3)$.}
\label{fig:quid-tri}
\end{figure}

Extend the leftmost arc segment at vertex $1$ and connect it with 
the inner boundary, thus creating a marked point at the inner boundary. 
See second picture in Figure~\ref{fig:quid-tri}. 

If \(m=1\), continue connecting the arc segments (from left to right) 
 at vertex $1$ to the inner boundary, until one arc is left, creating
$a_1-2$ marked points at the inner boundary. Then connect the remaining segment to the first marked point on the inner boundary.

If \(m>1\), continue connecting the arc segments (from left to right) 
at vertex $1$ to the inner boundary, creating
$a_1-1$ marked points at the inner boundary. 

Now assume the arc segments at vertices $1,2,\dots, i-1$ have been connected to the inner 
boundary. 
Connect the left most arc segment at 
vertex $i$ to the last of the marked points created for vertex $i-1$. If $a_i=2$, this is the only 
segment to connect. 
If $i<m$ {\em and if} there is an index $j$ with $i<j\le m$ such that $a_j>2$, 
create new vertices for the next $a_i-2$ arc segments at $i$. 
See third picture in Figure~\ref{fig:quid-tri}. 
If $i=m$ or if $a_j=2$ for all $j=i+1,\dots,m$, 
create new vertices for the next $a_i-3$ arc segments at $i$ and connect the 
last arc segment to the first marked point created for vertex $1$, as in the last picture of 
Figure~\ref{fig:quid-tri}. In the latter case, connect any remaining vertices $j$ with $j>i$ and 
$a_j=2$ with this same marked point. 


As we see from Figure~\ref{fig:first-choice-tri}, the triangulation is only dependent on the first arc: once the endpoint 
of the first arc on the inner boundary is fixed, the rest is determined. Thus we have the uniqueness of the triangulation up to this first choice. 
\end{proof}

\begin{figure}[ht!]
 \begin{subfigure}[b]{0.3\linewidth}
 \centering
\begin{tikzpicture}[scale=.4]
\fill[background] (0,0) rectangle (8,5) ; 
\draw[thick] (-.5,0) -- (8.5,0); 
\draw[thick] (-.5,5) -- (8.5,5); 

\node[inner sep=1.5pt, fill, circle] (3) at (1, 0) {};
\node[inner sep=1.5pt, fill, circle] (4) at (3, 0) {};
\node[inner sep=1.5pt, fill, circle] (1) at (5, 0) {};
\node[inner sep=1.5pt, fill, circle] (2) at (7, 0) {};
\node[below] (1txt) at (1) {$1$};
\node[below] (2txt) at (2) {$2$};
\node[below] (3txt) at (3) {$3$};
\node[below] (4txt) at (4) {$4$};

\node[inner sep=1.5pt, fill, circle] (4-bar) at (.5, 5) {};
\node[inner sep=1.5pt, fill, circle] (5-bar) at (2.5, 5) {};
\node[inner sep=1.5pt, fill, circle] (1-bar) at (4, 5) {};
\node[inner sep=1.5pt, fill, circle] (2-bar) at (5, 5) {};
\node[inner sep=1.5pt, fill, circle] (3-bar) at (7.5, 5) {};
\node[above] (1-bartxt) at (1-bar) {$\overline{1}$};
\node[above] (2-bartxt) at (2-bar) {$\overline{2}$};
\node[above] (3-bartxt) at (3-bar) {$\overline{3}$};
\node[above] (4-bartxt) at (4-bar) {$\overline{4}$};
\node[above] (5-bartxt) at (5-bar) {$\overline{5}$};

\draw[arc2, very thick] (1)--(1-bar);
\begin{scope}[arc1, very thick]
\draw	 (3) -- (4-bar);
\draw	 (4) -- (4-bar);
\draw	 (4) -- (5-bar);
\draw	 (4) -- (1-bar);
\draw	 (1) -- (2-bar);
\draw	 (2) -- (2-bar);
\draw	 (2) -- (3-bar);
\draw	 (2) -- (8,3);
\draw	 (0,3) -- (4-bar);
\end{scope}

\begin{scope}[very thick]
\clip (1) circle (.8cm) (2) circle (.8cm) (3) circle (.8cm) (4) circle (.8cm);
\draw (1)--(1-bar);
\draw	 (3) -- (4-bar);
\draw	 (4) -- (4-bar);
\draw	 (4) -- (5-bar);
\draw	 (4) -- (1-bar);
\draw	 (1) -- (2-bar);
\draw	 (2) -- (2-bar);
\draw	 (2) -- (3-bar);
\draw	 (2) -- (8,3);
\end{scope}
\end{tikzpicture}
\end{subfigure}
\begin{subfigure}[b]{0.3\linewidth}
\centering
\begin{tikzpicture}[scale=.4]
\fill[background] (0,0) rectangle (8,5) ; 
\draw[thick] (-.5,0) -- (8.5,0); 
\draw[thick] (-.5,5) -- (8.5,5); 

\node[inner sep=1.5pt, fill, circle] (3) at (1, 0) {};
\node[inner sep=1.5pt, fill, circle] (4) at (3, 0) {};
\node[inner sep=1.5pt, fill, circle] (1) at (5, 0) {};
\node[inner sep=1.5pt, fill, circle] (2) at (7, 0) {};
\node[below] (1txt) at (1) {$1$};
\node[below] (2txt) at (2) {$2$};
\node[below] (3txt) at (3) {$3$};
\node[below] (4txt) at (4) {$4$};

\node[inner sep=1.5pt, fill, circle] (5-bar) at (.5, 5) {};
\node[inner sep=1.5pt, fill, circle] (1-bar) at (2.5, 5) {};
\node[inner sep=1.5pt, fill, circle] (2-bar) at (4, 5) {};
\node[inner sep=1.5pt, fill, circle] (3-bar) at (5, 5) {};
\node[inner sep=1.5pt, fill, circle] (4-bar) at (7.5, 5) {};
\node[above] (1-bartxt) at (1-bar) {$\overline{1}$};
\node[above] (2-bartxt) at (2-bar) {$\overline{2}$};
\node[above] (3-bartxt) at (3-bar) {$\overline{3}$};
\node[above] (4-bartxt) at (4-bar) {$\overline{4}$};
\node[above] (5-bartxt) at (5-bar) {$\overline{5}$};

\draw[arc2, very thick] (1)--(1-bar);
\begin{scope}[arc1, very thick]
\draw	 (3) -- (0,3);
\draw	 (8,3) -- (4-bar);
\draw	 (4) -- (0,4);
\draw	 (8,4) -- (4-bar);
\draw	 (4) -- (5-bar);
\draw	 (4) -- (1-bar);
\draw	 (1) -- (2-bar);
\draw	 (2) -- (2-bar);
\draw	 (2) -- (3-bar);
\draw	 (2) -- (4-bar);
\end{scope}
\begin{scope}[very thick]
\clip (1) circle (.8cm) (2) circle (.8cm) (3) circle (.8cm) (4) circle (.8cm);
\draw (1)--(1-bar);
\draw	 (3) -- (0,3);
\draw	 (4) -- (0,4);
\draw	 (4) -- (5-bar);
\draw	 (4) -- (1-bar);
\draw	 (1) -- (2-bar);
\draw	 (2) -- (2-bar);
\draw	 (2) -- (3-bar);
\draw	 (2) -- (4-bar);
\end{scope}
\end{tikzpicture}
\end{subfigure}
 \begin{subfigure}[b]{0.3\linewidth}
 \centering
\begin{tikzpicture}[scale=.4]
\fill[background] (0,0) rectangle (8,5) ; 
\draw[thick] (-.5,0) -- (8.5,0); 
\draw[thick] (-.5,5) -- (8.5,5); 

\node[inner sep=1.5pt, fill, circle] (3) at (1, 0) {};
\node[inner sep=1.5pt, fill, circle] (4) at (3, 0) {};
\node[inner sep=1.5pt, fill, circle] (1) at (5, 0) {};
\node[inner sep=1.5pt, fill, circle] (2) at (7, 0) {};
\node[below] (1txt) at (1) {$1$};
\node[below] (2txt) at (2) {$2$};
\node[below] (3txt) at (3) {$3$};
\node[below] (4txt) at (4) {$4$};

\node[inner sep=1.5pt, fill, circle] (1-bar) at (.5, 5) {};
\node[inner sep=1.5pt, fill, circle] (2-bar) at (2.5, 5) {};
\node[inner sep=1.5pt, fill, circle] (3-bar) at (4, 5) {};
\node[inner sep=1.5pt, fill, circle] (4-bar) at (5, 5) {};
\node[inner sep=1.5pt, fill, circle] (5-bar) at (7.5, 5) {};
\node[above] (1-bartxt) at (1-bar) {$\overline{1}$};
\node[above] (2-bartxt) at (2-bar) {$\overline{2}$};
\node[above] (3-bartxt) at (3-bar) {$\overline{3}$};
\node[above] (4-bartxt) at (4-bar) {$\overline{4}$};
\node[above] (5-bartxt) at (5-bar) {$\overline{5}$};

\draw[arc2, very thick] (1)--(1-bar);
\begin{scope}[arc1, very thick]
\draw	 (3) -- (0,1);
\draw	 (8,1)--(4-bar);
\draw	 (4) -- (0,2.5);
\draw	 (8,2.5)--(4-bar);
\draw	 (4) -- (0,4);
\draw	 (8,4)--(5-bar);
\draw	 (4) -- (1-bar);
\draw	 (1) -- (2-bar);
\draw	 (2) -- (2-bar);
\draw	 (2) -- (3-bar);
\draw	 (2) -- (4-bar);
\end{scope}

\begin{scope}[very thick]
\clip (1) circle (.8cm) (2) circle (.8cm) (3) circle (.8cm) (4) circle (.8cm);
\draw	 (1)--(1-bar);
\draw	 (3) -- (0,1);
\draw	 (4) -- (0,2.5);
\draw	 (4) -- (0,4);
\draw	 (4) -- (1-bar);
\draw	 (1) -- (2-bar);
\draw	 (2) -- (2-bar);
\draw	 (2) -- (3-bar);
\draw	 (2) -- (4-bar);
\end{scope}
\end{tikzpicture}
\end{subfigure}

\begin{subfigure}[b]{0.3\linewidth}
\centering
\begin{tikzpicture}[scale=.4]
\draw[thick,fill=background, radius=4cm] (0,0) circle; 
\draw[thick,fill=white!10, radius=2cm] (0,0) circle; 
\node[inner sep=1.5pt, fill, circle] (1) at (0:4) {};
\node[below right] (1txt) at (1) {$\tiny\circled{3}$};

\node[inner sep=1.5pt, fill, circle] (2) at (90:4) {};
\node[above right] (1txt) at (2) {$\tiny\circled{4}$};

\node[inner sep=1.5pt, fill, circle] (3) at (180:4) {};
\node[below left] (1txt) at (3) {$\tiny\circled{2}$};

\node[inner sep=1.5pt, fill, circle] (4) at (270:4) {};
\node[below right] (1txt) at (4) {$\tiny\circled{4}$};

\node[right] (1txt) at (1) {$1$};
\node[above] (2txt) at (2) {$2$};
\node[left] (3txt) at (3) {$3$};
\node[below] (4txt) at (4) {$4$};
\node[inner sep=1.5pt, fill, circle] (1-bar) at (0:2) {};
\node[inner sep=1.5pt, fill, circle] (2-bar) at (70:2) {};
\node[inner sep=1.5pt, fill, circle] (3-bar) at (130:2) {};
\node[inner sep=1.5pt, fill, circle] (4-bar) at (190:2) {};
\node[inner sep=1.5pt, fill, circle] (5-bar) at (290:2) {};
\node[left] (1-bartxt) at (1-bar) {$\overline{1}$};
\node[below] (2-bartxt) at (2-bar) {$\overline{2}$};
\node[below] (3-bartxt) at (3-bar) {$\overline{3}$};
\node[right] (4-bartxt) at (4-bar) {$\overline{4}$};
\node[above] (5-bartxt) at (5-bar) {$\overline{5}$};

\begin{scope}[color=arc1, very thick]
\draw	 (3) to (4-bar);
\draw	 (4) to[bend left] (4-bar);
\draw	 (4) to (5-bar);
\draw	 (4) to[bend right] (1-bar);
\draw	 (1) to[bend right] (2-bar);
\draw	 (2) to (2-bar);
\draw	 (2) to (3-bar);
\draw	 (2) .. controls (-3,2) and (-2.5,1) .. (4-bar);
\end{scope}
\draw[arc2, very thick] (1) to (1-bar);

\begin{scope}[even odd rule, very thick]
\clip (1) circle (.8cm) (2) circle (.8cm) (3) circle (.8cm) (4) circle (.8cm);
\draw	 (3) to (4-bar);
\draw	 (4) to[bend left] (4-bar);
\draw	 (4) to (5-bar);
\draw	 (4) to[bend right] (1-bar);
\draw	 (1) to[bend right] (2-bar);
\draw	 (2) to (2-bar);
\draw	 (2) to (3-bar);
\draw	 (2) .. controls (-3,2) and (-2.5,1) .. (4-bar);
\draw[very thick] (1) to (1-bar);
\end{scope}

\end{tikzpicture}
\end{subfigure}
%
%
\begin{subfigure}[b]{0.3\linewidth}
\centering
\begin{tikzpicture}[scale=.4]
\draw[thick,fill=background, radius=4cm] (0,0) circle; 
\draw[thick,fill=white!10, radius=2cm] (0,0) circle; 

\node[inner sep=1.5pt, fill, circle] (1) at (0:4) {};
\node[below right] (1txt) at (1) {$\tiny\circled{3}$};

\node[inner sep=1.5pt, fill, circle] (2) at (90:4) {};
\node[above right] (1txt) at (2) {$\tiny\circled{4}$};

\node[inner sep=1.5pt, fill, circle] (3) at (180:4) {};
\node[below left] (1txt) at (3) {$\tiny\circled{2}$};

\node[inner sep=1.5pt, fill, circle] (4) at (270:4) {};
\node[below right] (1txt) at (4) {$\tiny\circled{4}$};

\node[right] (1txt) at (1) {$1$};
\node[above] (2txt) at (2) {$2$};
\node[left] (3txt) at (3) {$3$};
\node[below] (4txt) at (4) {$4$};

\node[inner sep=1.5pt, fill, circle] (2-bar) at (0:2) {};
\node[inner sep=1.5pt, fill, circle] (3-bar) at (70:2) {};
\node[inner sep=1.5pt, fill, circle] (4-bar) at (130:2) {};
\node[inner sep=1.5pt, fill, circle] (5-bar) at (190:2) {};
\node[inner sep=1.5pt, fill, circle] (1-bar) at (290:2) {};
\node[above] (1-bartxt) at (1-bar) {$\overline{1}$};
\node[left] (2-bartxt) at (2-bar) {$\overline{2}$};
\node[below] (3-bartxt) at (3-bar) {$\overline{3}$};
\node[below] (4-bartxt) at (4-bar) {$\overline{4}$};
\node[right] (5-bartxt) at (5-bar) {$\overline{5}$};

\begin{scope}[color=arc1, very thick]
\draw	 (3) to[bend left] (4-bar);
\draw	 (4) .. controls (230:3.9) and (180:3.8).. (4-bar);
\draw	 (4) to[bend left] (5-bar);
\draw	 (4) to[bend left] (1-bar);
\draw	 (1) to[bend left] (2-bar);
\draw	 (2) to[bend left] (2-bar);
\draw	 (2) to[bend left] (3-bar);
\draw	 (2) to[bend left] (4-bar);
\end{scope}
\draw[arc2, very thick] (1) to[bend left] (1-bar);

\begin{scope}[even odd rule, very thick]
\clip (1) circle (.8cm) (2) circle (.8cm) (3) circle (.7cm) (4) circle (.8cm);
\draw	 (3) to[bend left] (4-bar);
\draw	 (4) .. controls (230:3.9) and (180:3.8).. (4-bar);
\draw	 (4) to[bend left] (5-bar);
\draw	 (4) to[bend left] (1-bar);
\draw	 (1) to[bend left] (2-bar);
\draw	 (2) to[bend left] (2-bar);
\draw	 (2) to[bend left] (3-bar);
\draw	 (2) to[bend left] (4-bar);
\draw 	 (1) to[bend left] (1-bar);
\end{scope}
\end{tikzpicture}
\end{subfigure}
%
%
\begin{subfigure}[b]{0.3\linewidth}
\centering
\begin{tikzpicture}[scale=.4]
\draw[thick,fill=background, radius=4cm] (0,0) circle; 
\draw[thick,fill=white!10, radius=2cm] (0,0) circle; 
\node[inner sep=1.5pt, fill, circle] (1) at (0:4) {};
\node[below right] (1txt) at (1) {$\tiny\circled{3}$};
\node[inner sep=1.5pt, fill, circle] (2) at (90:4) {};
\node[above right] (1txt) at (2) {$\tiny\circled{4}$};
\node[inner sep=1.5pt, fill, circle] (3) at (180:4) {};
\node[below left] (1txt) at (3) {$\tiny\circled{2}$};
\node[inner sep=1.5pt, fill, circle] (4) at (270:4) {};
\node[below right] (1txt) at (4) {$\tiny\circled{4}$};

\node[right] (1txt) at (1) {$1$};
\node[above] (2txt) at (2) {$2$};
\node[left] (3txt) at (3) {$3$};
\node[right] (4txt) at (4) {$4$};
\node[inner sep=1.5pt, fill, circle] (3-bar) at (0:2) {};
\node[inner sep=1.5pt, fill, circle] (4-bar) at (70:2) {};
\node[inner sep=1.5pt, fill, circle] (5-bar) at (130:2) {};
\node[inner sep=1.5pt, fill, circle] (1-bar) at (190:2) {};
\node[inner sep=1.5pt, fill, circle] (2-bar) at (290:2) {};
\node[right] (1-bartxt) at (1-bar) {$\overline{1}$};
\node[above] (2-bartxt) at (2-bar) {$\overline{2}$};
\node[left] (3-bartxt) at (3-bar) {$\overline{3}$};
\node[below] (4-bartxt) at (4-bar) {$\overline{4}$};
\node[below] (5-bartxt) at (5-bar) {$\overline{5}$};

\begin{scope}[color=arc1,very thick]
\draw	 (3) .. controls (140:4) and (100:3.5) .. (4-bar);
\draw	 (4) .. controls (250:4) and (210:3.8) .. (180:3.3) .. controls (150:3.3) and (120:3.3)..  (4-bar);
\draw	 (4) ..controls (225:4) and (180:3.5) .. (5-bar);
\draw	 (4) to[bend left] (1-bar);
\draw	 (1) to[bend left] (2-bar);
\draw	 (2) .. controls (45:4) and (350:4) .. (2-bar);
\draw	 (2) to[bend left] (3-bar);
\draw	 (2) to[bend left] (4-bar);
\end{scope}
\draw[arc2, very thick] (1) .. controls (315:5) and (250:4)..  (1-bar);

\begin{scope}[even odd rule, very thick]
\clip (1) circle (.8cm) (2) circle (.8cm) (3) circle (.7cm) (4) circle (.8cm);
\draw	 (3) .. controls (140:4) and (100:3.5) .. (4-bar);
\draw	 (4) .. controls (250:4) and (210:3.8) .. (180:3.2) .. controls (150:3.2) and (120:3.3)..  (4-bar);
\draw	 (4) ..controls (225:4) and (180:3.5) .. (5-bar);
\draw	 (4) to[bend left] (1-bar);
\draw	 (1) to[bend left] (2-bar);
\draw	 (2) .. controls (45:4) and (350:4) .. (2-bar);
\draw	 (2) to[bend left] (3-bar);
\draw	 (2) to[bend left] (4-bar);
\draw 	(1) .. controls (315:5) and (225:4)..  (1-bar);
\end{scope}
\end{tikzpicture}
\end{subfigure}
\caption{An illustration of the uniqueness of the triangulation up to the choice of the first bridging arc (in orange), 
for the quiddity sequence \(q=(3,4,2,4)\).
}
\label{fig:first-choice-tri}
\end{figure}
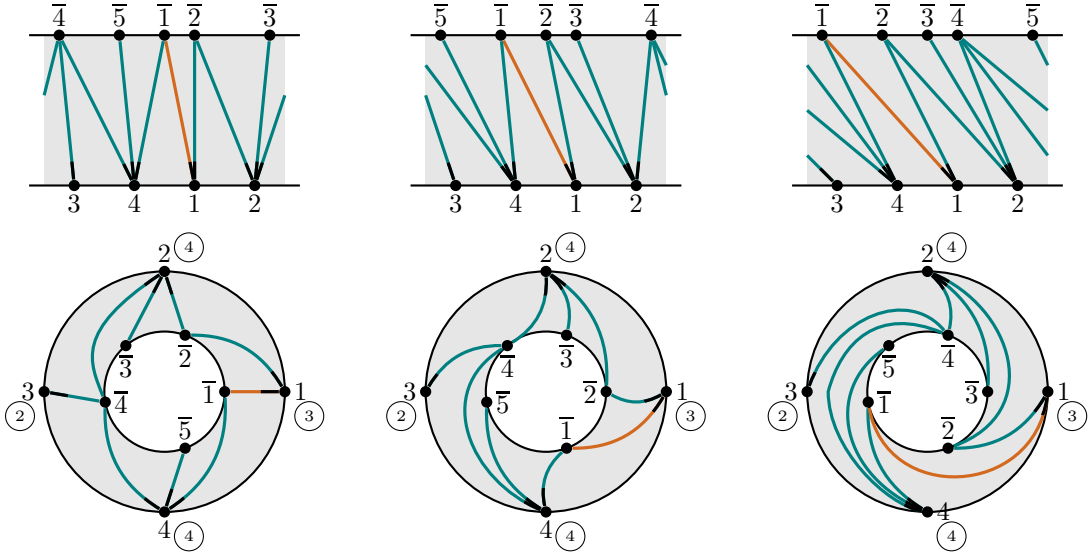

The above proposition also suggests that if $q_1$ is a skeletal quiddity sequence corresponding to the outer boundary in the skeletal triangulation $\cT$, then it determines the skeletal quiddity sequence $q_2$ for the inner boundary as well. We state the formula in the next corollary. For this, we write the entries of a skeletal quiddity sequence $q$ as follows:
\[
q=(a_{i_1}, \underbrace{ 2, \ldots , 2}_{k_1}, a_{i_2},\underbrace{ 2, \ldots , 2}_{k_2},\dots, 
a_{i_r},\underbrace{ 2, \ldots , 2}_{k_r})
\] 
or $(a_{i_1},2^{(k_1)},a_{i_2},2^{(k_2)},\dots, a_{i_r},2^{(k_r)})$ for short, 
where $a_{i_j}> 2$ and $k_j\ge 0$ 
for all $j$.

\begin{corollary}\label{cor:partner-quiddity}
Let $q_1=(a_{i_1},2^{(k_1)},a_{i_2},2^{(k_2)},\dots, a_{i_r},2^{(k_r)})$ 
be a skeletal quiddity sequence, where 
$a_{i_j}>2$ and $k_j\ge 0$ for all $j$. Let $\cT$ be the triangulation associated to $q_1$ 
by Proposition~\ref{prop:frieze-tria}. Then the quiddity sequence given by the inner boundary of 
$\cT$ is 
$q_2=(2^{(a_{i_1}-3)},k_1+3, 2^{(a_{i_2}-3)}, k_2+3,\dots, 2^{(a_{i_r}-3)},k_r+3)$. 
\end{corollary}


The following proposition describes the correspondence between triangulations and infinite periodic friezes, and their respective skeletal versions. 
Let $\mathcal{T}$ be a triangulation of an annulus. Let $f$ be the map that sends $\mathcal T$ to 
the corresponding pair of infinite periodic friezes $f(\mathcal T)=(\mathcal F_1, \mathcal F_2)$. 
Let $s_T$ be the map that sends a triangulation to its skeletal triangulation. 
Let $s_F$ be the map sending a pair of infinite periodic friezes to the pair of their skeletal versions, 
so that $s_F((\mathcal F_1, \cF_2))=(\mathcal F_1^s,\mathcal F_2^s)$.

\begin{proposition}\label{skeleton}
We have that $f\circ s_T=s_F\circ f$, i.e. $f(\mathcal T^s)=s_F\circ f(\mathcal T)$.

\begin{center}
\begin{tikzpicture}
\node (1) at (0,0) {$\cT$};
\node (2) at (0,-1.5) {$\cT^s$};
\node (3) at (2,0) {($\cF_1$,};
\node (4) at (2.7,0) {$\cF_2$)};
\node (7) at (2.3,-1.6) {,};
\node (5) at (2,-1.5) {($\cF_1^s$};
\node (6) at (2.7,-1.5) {${\cF_2^s}$)};

\draw [|->] (1) --(2);
\draw [| ->](1) -- (3) ;
\draw [|->] (2)--(5);
\draw [|->] (3)--(5);
\draw [|->] (4)--(6);
\end{tikzpicture}
\end{center}
\end{proposition}
\begin{proof}
First notice that the reduction $s_T$ of a triangulation $\cT$ to its skeletal triangulation $\cT^s$ consists of 
cutting along all bounding arcs and removing the triangulated polygons attached with them by Definition \ref{s_T}.
This process is well defined by Lemma \ref{lm:skeleton-tri}. So it is enough to show that removing one polygon from a triangulation of an annulus produces the same pair of infinite periodic friezes as are obtained by applying reduction process to the friezes as in Definition \ref{frieze reduction}.

Consider now the deletion of a triangulated n-gon from a triangulation of an annulus, where \(n\geq 3\). We will show that such a deletion corresponds to a sequence of reductions on the corresponding quiddity sequence. If \(n=3\),
we use the fact that
 reduction of a quiddity sequence at a 1 corresponds to deletion of a peripheral triangle 
(as shown in Fig~\ref{fig:skeleton-triangulation}) in the corresponding triangulation of an annulus 
by the Remark \ref{rem:reduce-triangles}. 
So in this case, deletion trivially corresponds to a sequence of reductions. For higher \(n\) we can show the statement by induction on \(n\). 

 Suppose the statement holds for \(n=k\). Suppose we remove a triangulated \(k+1\)-gon. The triangulated \(k+2\)-gon must have at least two peripheral triangles, and at least one of them must also be on the boundary of the original triangulated annulus. If we delete this peripheral triangle from the annulus, our original \(k+1\)-gon is now a \(k\)-gon. Looking at the quiddity sequence side, we perform one reduction when deleting the peripheral triangle, and then a series of reductions when deleting the \(k-gon\). Hence the deletion of a \(k+1\)-gon corresponds to a series of reductions. 
\end{proof}

\begin{lemma}
\label{lem:skeletal-property-2}
Let $\cT$ be a triangulation of an annulus, let $(q_1,q_2)$ the pair 
of quiddity sequences associated to $\cT$. The following are equivalent:
\begin{enumerate}
\item $\cT=\cT^s$. 
\item There are no 1's in $q_1$ and $q_2$. 
\end{enumerate}
\end{lemma}

\begin{proof} 
Suppose $\cT=\cT^s$, then there are no peripheral arcs in the triangulation and hence all 
entries in $q_1$ and $q_2$ are $\ge 2$. 

Suppose conversely that 
\(q_1\) and \(q_2\) have no 1's. When we do the construction of Proposition \ref{prop:frieze-tria}, this 
means that every marked point on each boundary will have at least one arc incident to it. Hence there can 
be no peripheral arcs.
\end{proof}

%
\subsection{Quivers to quiddity sequences and back}\label{sec:quiv-quid}

In this section, we explain how any non-oriented cyclic quiver gives rise to a pair of quiddity sequences of an 
infinite periodic frieze and hence to a triangulation of an annulus. Let $Q$ be a non-oriented cyclic quiver, with vertices $1,2,\dots, n$ labeled anti-clockwise around the cycle. 
We associate a pair of quiddity sequences to $Q$ as follows. 

First observe that in any non-oriented cycle there is at least one source and at least one 
sink. Without loss of generality we can assume that the vertex $1$ is a source. 
To define the quiddity sequences we distinguish between {\em increasing arrows} 
(arrows $i\to i+1$, including $n\to 1$) 
and {\em decreasing arrows} (arrows $i\leftarrow i+1$, including $1\to n$) in $Q$. 
Let $p$ be the number of 
decreasing arrows in $Q$ and let $t=n-p$ be the number of increasing arrows in $Q$. 
We define an increasing path as a path of increasing arrows. It is maximal when it starts at a source and ends at a sink. Similarly, a decreasing path is a path of decreasing arrows. 


\begin{definition}[Quiddity sequences from a non-oriented 
 cyclic quiver
$Q$]\label{def:quiv-quid}
Let $Q$ be a non-oriented cyclic quiver with vertices $\{1,2,\dots,n\}$. Assume that vertex $1$ is a source. \\
Let $j_1=1$ and $\{j_2,\dots, j_p\}$ be the set of tails of the decreasing arrows 
in $Q$ for $1\le p<n$, with $1<j_2<\dots<j_p\leq n$. 
Let $c_1$ be the length of the maximal (increasing) linear path starting with 
$1\to 2$. For $k=2,\dots, p$, let $c_{j_k}\ge 0$ be the length of the maximal increasing linear path 
starting 
at $j_k$. Define:
$$\sigma(Q):=(a_1,\dots, a_p) \text{ where } a_k:=c_{j_k}+2, \text{ for } k=1,\dots, p.$$
Similarly, let $\{m_1,m_2,\dots, m_t\}$ be the set of heads of the increasing arrows in $Q$, $1\le t <n$, 
with $2=m_1<m_2< \cdots <m_t\le n$. For $k=1,\dots, t$, let $d_{j_k}\ge 0$ be the length of the maximal decreasing path ending at $m_{j_k}$. Define:
$$ \widetilde{\sigma}(Q):=(b_1,\dots, b_t) \text{ where } b_k:=d_{j_k}+2, \text{ for } k=1,\dots, p.$$
\end{definition}

Note that since we assume that vertex 1 is a source, $c_1 > 0$ and $d_{j_t} >0$. 

\begin{remark}\label{rem:quiv-quid}
Let $Q$ be a non-oriented cyclic quiver and consider $\sigma(Q)$ and $\widetilde{\sigma}(Q)$ from 
Definition~\ref{def:quiv-quid}. We have $a_j\ge 2$ for all $j$ and $b_i\ge 2$ for all $i$. Furthermore, $a_1>2$ and 
$b_{t}>2$. So by result~\cite[Corollary 2.2]{bpt}, both are quiddity sequences of infinite friezes. 
\end{remark}


\begin{example}
Consider the quiver from Figure~\ref{fig:quiv-quid}. Here, $\{j_1,\dots, j_p\}=\{1,4,6,7,8\}$ 
and $\{m_1,\dots, m_t\}=\{2,3,5,9\}$. From this, we get $c_{j_1} = c_1 =2$, $c_{j_2}=1$, $c_{j_3}=0$, 
$c_{j_4}=0$, $c_{j_5}=1$ and 
$d_{m_1}=0$, $d_{m_2}=1$, $d_{m_3}=3$, $d_{m_4}=1$. And so 
$\sigma(Q)=(4,3,2,2,3)$ and $\widetilde{\sigma}(Q)=(2,3,5,3)$.
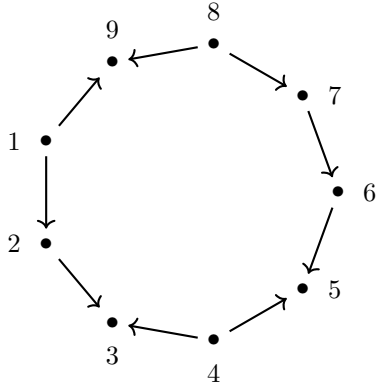
\begin{figure}
\begin{tikzpicture}
    \node[label=left:1] (1) at (160:2) {$\bullet$};
    \node[label=left:2] (2) at (200:2) {$\bullet$};
    \node[label=below:3] (3) at (240:2) {$\bullet$};
    \node[label=below:4] (4) at (280:2) {$\bullet$};
    \node[label=right:5] (5) at (3200:2) {$\bullet$};
    \node[label=right:6] (6) at (0:2) {$\bullet$};
    \node[label=right:7] (7) at (40:2) {$\bullet$};
    \node[label=above:8] (8) at (80:2) {$\bullet$};
    \node[label=above:9] (9) at (120:2) {$\bullet$};
    
    \draw[->, thick] (1) to (2);
    \draw[<-, thick]  (3) to (2);
    \draw[<-, thick]  (3) to (4);
    \draw[->, thick]  (4) to (5);
    \draw[<-, thick]  (5) to (6);
    \draw[->, thick]  (7) to (6);
    \draw[->, thick]  (8) to (7);
    \draw[->, thick]  (8) to (9);
    \draw[->, thick]  (1) to (9);
  \end{tikzpicture}
\caption{A non-oriented cycle on 9 vertices.}\label{fig:quiv-quid}
\end{figure}
\end{example}

\begin{lemma} \label{lem:head-or-tail} 
Each vertex in a non-oriented cyclic quiver is either the head of an increasing arrow or the tail of a decreasing arrow (it cannot be both). 
\end{lemma}
\begin{proof}
If a vertex $k$ in a non-oriented cyclic quiver $Q$ is a sink or a source, then it is clear that it is the head of an increasing arrow or the tail of a decreasing arrow (respectively). If the vertex $k$ is neither a sink nor a source, then it is one of the following: \[ k-1 \longrightarrow k \longrightarrow k+1 \]
\[ k-1 \longleftarrow k \longleftarrow k+1 \]
In the first case $k$ is the head of an increasing arrow, and in the second case, it is the tail of a decreasing arrow. 
To see that it cannot be both, notice that if it is both, we get a two-cycle at that vertex. This contradicts the assumption that $Q$ is a non-oriented cyclic quiver.
\end{proof}

\begin{lemma}\label{lm:properties-sigma}
Let $Q$ be a non-oriented cyclic quiver with vertices $\{1,2,\dots,n\}$. Assume that vertex $1$ is a source.
Let $\sigma(Q)=(a_1,\dots, a_p)$ and $\widetilde{\sigma}(Q)=(b_1,\dots, b_t)$ be as in Definition \ref{def:quiv-quid}. Then:
\begin{enumerate}
\item 
$\sigma(Q)$
and $\widetilde{\sigma}(Q)$
are skeletal, non-trivial quiddity sequences, i.e. $a_i>1$ and $b_j> 1$ for all $i,j$ and not all elements of \(\sigma(Q)\) and \(\widetilde{\sigma}(Q)\) are equal to $2$,
\item 
$\sum a_i + \sum b_j=3n.$
\end{enumerate}
\end{lemma}

\begin{proof}
By Remark~\ref{rem:quiv-quid}, $\sigma(Q)=(a_1,\dots, a_p)$ and $\widetilde{\sigma}(Q)=(b_1,\dots, b_t)$
are quiddity sequences giving rise to infinite periodic friezes. 
By construction, the $a_i$ and the $b_j$ are all greater than or equal to $2$. Since $Q$ is not cyclic, 
at least one $c_{j_k}$ and at least one $d_{j_k}$ is positive, so (1) holds. 

To see that (2) holds: 
The $c_{j_i}$ count the decreasing arrows and the $d_{m_k}$ the increasing arrows in $Q$.
Since we add $2$ for every arrow to get $a_i$ and $b_j$, the total sum is $3n$. 
\end{proof}

\begin{definition}[A non-oriented cyclic quiver from a quiddity sequence]\label{def:quidd-cycle}
Let \[q=(a_{j_1},2^{(k_1)},a_{j_2},2^{(k_2)},\dots, a_{j_r},2^{(k_r)})\] be a skeletal quiddity sequence
where $a_{j_i}> 2$ and $k_i\ge 0$ 
for all $i=1, \ldots, r$. 
To $q$, we associate a cyclic quiver $Q=\mu(q)$ as follows: 
Starting with a source at vertex $1$, $Q$ has $a_{j_1} - 2$ increasing arrows, then a sink, 
then $k_1+1$ decreasing arrows, then a source, followed by $a_{j_2} -2$ increasing arrows, 
another sink, then $k_2+1$ decreasing arrows, etc. This ends with 
$k_r+1$ decreasing arrows, the last of them being $n\leftarrow 1$. 
\end{definition}

\begin{theorem} \label{thm:quiver to quiddity}
Let $n\ge 2$. There is a bijection between the following classes.
\begin{enumerate}
\item
Unlabeled non-oriented cyclic quivers with $n$ vertices.
\item
Skeletal 
quiddity sequences $q=(a_1,\dots, a_p)$ with \[n=p + \sum_{k=1}^p(a_k-2)\] for some 
$1\le p<n$. 
\end{enumerate}
\end{theorem}

\begin{proof}
We use the maps $\sigma$ and $\mu$ to get the bijection. 
The map $\sigma$ associates to every non-oriented cyclic quiver $Q$ a 
skeletal quiddity sequence. 
We show that the image $\sigma(Q)$ is of the form described in (2). 
Assume that $Q$ has a source at vertex 1 and consider $\sigma(Q)=(a_1,\dots, a_p)$. 
Recall from Definition~\ref{def:quiv-quid} that $a_k=c_{j_k}+2$, where 
$c_{j_k}$ counts the length of the maximal increasing path starting at $j_k$, 
$a_i\ge 2$ for all $i$, and $a_1>2$. 
So $(a_1,\dots, a_p)$ is a skeletal non-trivial quiddity sequence. Also, $a_k=c_{j_k}+2$ implies that the right hand side of the equality in (2) becomes:
\[
p + \sum_{k=1}^p (a_k-2) = p +\sum_{k=1}^p c_{j_k}. 
\] 
Let us denote the maximal increasing path starting at $j_k$ by $P_k$ 
(this can be of length 0). All these paths are disjoint (they do not share vertices). By Lemma~ \ref{lem:head-or-tail}, each vertex in $Q$ is either the head of an increasing arrow or the tail of a decreasing arrow (it cannot be both). If it is the tail 
of a decreasing arrow, then it is the tail of some $P_k$. If it is the head of an increasing arrow, then it is either 
the head of a path $P_k$ or an internal vertex of $P_k$, for some $k$. In any case, the vertex belongs to a 
maximal increasing path. This implies that each vertex of $Q$ appears exactly once in one $P_k$. The number of vertices in $P_k$ is $c_{j_k}+1$. Hence we have 
\[\sum_{k=1}^p (c_{j_k}+1)= p + \sum_{k=1}^p c_{j_k} = n.\]

Now we argue that $\mu(\sigma(Q))=Q$. Let $Q$ be a non-oriented cyclic quiver described as in the beginning of Section 3.2 such that $\{ j_1=1, j_2,\dots, j_p\}$ is the set of tails of its decreasing arrows. Then $\sigma(Q)=(a_1,\dots, a_p)$ is the quiddity sequence from Definition~\ref{def:quiv-quid}.

To prove that $\mu(\sigma(Q))$ has $n$ vertices, we use induction on the length of the quiddity sequence. Suppose $\sigma(Q)=(a_1)$, where $a_1 >2$. We claim that $n=a_1-1$. The quiver $\mu(\sigma(Q))$ has a source at vertex 1, then an increasing path of length $a_1-2$, then a sink and a decreasing arrow from 1 to that sink. So the number of vertices of $\mu(\sigma(Q))$ is $a_1-1$. This is shown in Figure~\ref{Basecase}.

\begin{figure}[h]
\begin{tikzpicture}[scale=.6]
    \node[shape=circle,fill=black, scale=0.4,label=left:1] (1) at (180:2) {};
    \node[shape=circle,fill=black, scale=0.4,label=below:2] (2) at (240:2) {};
    \node[shape=circle,fill=black, scale=0.4,label=below:3] (3) at (300:2) {};
    \node[shape=circle,fill=black, scale=0.4] (4) at (0:2) {};
    \node[shape=circle,fill=black, scale=0.4] (5) at (60:2) {};
    \node[shape=circle,fill=black, scale=0.4,label=above:$a_1-1$] (6) at (120:2) {};
    
    \draw[->, thick] (1) to (2);
    \draw[<-, thick]  (3) to (2);
    \draw[->, thick, dotted]  (3) to (4);
    \draw[->, thick, dotted]  (4) to (5);
    \draw[->, thick]  (5) to (6);
    \draw[->, thick]  (1) to (6);
  \end{tikzpicture}
  \caption{The quiver $\mu(\sigma(Q))$ where $\sigma(Q)=(a_1)$.}
  \label{Basecase}
  \end{figure}
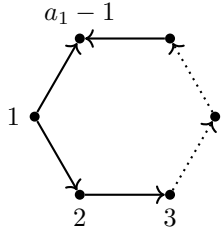
  
Let $\sigma(Q)=(a_1,\dots, a_{\ell+1})$. By induction, the quiddity sequence $q=(a_1,\dots, a_{\ell})$ is such that $\mu(q)$ is a quiver with $n$ vertices fulfilling \[n=\ell + \sum_{k=1}^{\ell}(a_k-2).\] If $a_{\ell+1}=2$, then the algorithm in Def~\ref{def:quidd-cycle} adds one decreasing arrow and hence one more vertex to the quiver $\mu(q)$. That means the total number of vertices is $n+1$, which is indeed the case since $a_{\ell+1}-2=0$ and
\[\ell +1 + \sum_{k=1}^{\ell+1}(a_k-2) =\ell +1 + \sum_{k=1}^{\ell}(a_k-2) = n +1. \]
If $a_{\ell+1} \geq 3$, then by Def~\ref{def:quidd-cycle}, we must add $a_{\ell+1}-2$ increasing arrows and a decreasing arrow from vertex 1 to the last vertex. This implies that we add $a_{\ell+1}-2$ vertices for the increasing arrows and one more for the decreasing arrow. 
 
\end{proof}

\begin{corollary} \label{cor:quiddity to triangulation}
Let $n\ge 2$. There is a bijection between the following classes.
\begin{enumerate}
\item
Skeletal 
quiddity sequences $q=(a_1,\dots, a_p)$ with $n=p + \sum_k (a_k-2)$, for some 
$1\le p<n$. 
\item
Skeletal triangulations of annuli $C_{m,n-m}$ with $1\le m< n$. 
\end{enumerate}
\end{corollary}

\begin{proof}
By Theorem~\ref{thm:quiver to quiddity}, (1) is equivalent to the associated quiver being non-oriented cycle 
and (by Lemma ~\ref{lm:skeletal-property-1}) these in turn correspond to triangulations given by bridging arcs, i.e. skeletal 
triangulations. 
\end{proof}


\begin{theorem}\label{thm:unique-frieze} Let $\cF^s_1$ be the skeletal frieze of an infinite periodic frieze $\cF_1$. Then
\begin{enumerate}
\item $\cF^s_1$ uniquely determines an infinite skeletal frieze $\cF^s_2$ such that the two 
form a pair of infinite periodic friezes associated to a triangulation of an annulus,
\item 
$\cF^s_1$ gives rise to an infinite family of infinite periodic friezes $\cF_2$ for each of which 
$(\cF^s_1,\cF_2)$ is a pair of infinite periodic friezes arising from a triangulation of an annulus.
\end{enumerate}
\end{theorem}
\begin{proof} $\qquad$
\begin{enumerate}
\item 
Let $q^s_1$ be the quiddity sequence of $\cF^s_1$. Then $q^s_1$ is a skeletal quiddity sequence. 
By Corollary~\ref{cor:quiddity to triangulation}, the quiddity sequence $q^s_1$ gives rise to a 
unique skeletal triangulation $\cT^s$ of a marked annulus. Then $\cT^s$ uniquely gives rise to a 
pair of infinite skeletal friezes one of which is $\cF^s_1$ and the other is $\cF^s_2$.
\item 
By Remark~\ref{rmk:reverse reduction}, we can apply reverse reduction process on $q^s_2$ repeatedly 
to obtain infinitely many friezes $\cF_2$.
\end{enumerate}
\end{proof}

%
\section{Growth coefficients}\label{sec:growth}

The entries in any infinite (periodic) frieze continue to grow. In fact, there is an invariant of periodic friezes 
which governs their growth, called the growth coefficient. It is known by \cite[Theorem 3.4]{bfpt} that the pair of infinite 
friezes associated with a triangulation have the same growth coefficient. In this section, we calculate 
the growth coefficients in terms of quiddity sequences and study a module theoretic interpretation of them.

%
\subsection{Properties of growth coefficients}

Let $\mathcal F$ be an infinite periodic frieze of (minimal) 
period $n$, i.e. $\mathcal F$ is $n$-periodic but there is no $n'<n$ such that 
$\mathcal F$ is $n'$-periodic. 
Set $s_r:=a_{1,rn}-a_{2,rn-1}$ for any $r>0$ (we follow the notation from Figure \ref{infinitefrieze}, see also Figure~\ref{fig:infinitefrieze}). 
We define the \emph{growth coefficient of $\cF$} as $s_{\mathcal F}:=s_1$. Recall that the first 
non-trivial row of a frieze is the quiddity row and we count rows onwards from there.

\begin{figure}[ht!]
\begin{center}
\begin{tikzpicture}[scale=.8, inner sep=4pt]
\node (00) at (-6,6) {$\ldots$};
\node (01) at (-4,6) {0};
\node (02) at (-2,6) {0};
\node (ai+) at (0,6) {$\ldots$};
\node (03) at (2,6) {0};
\node (04) at (4,6) {0};
\node (05) at (6,6) {$\ldots$};
\node (10) at (-7,5) {$\ldots$};
\node (11) at (-5,5) {1};
\node (12) at (-3,5) {1};
\node (ai+) at (0,5) {$\ldots$};
\node (12) at (-1,5) {1};
\node (12) at (1,5) {1};
\node (14) at (3,5) {1};
\node (15) at (5,5) {1};
\node (16) at (7,5) {$\ldots$};

\node (a-2) at (-6,4) {$\ldots$};
\node (a-1) at (-4,4) {\(a_{i}\)};
\node (a0) at (-2,4) {\(a_{i+1}\)};
\node (a1) at (0,4) {$\hdots$};
\node (a2) at (2,4) {\(a_{rn-1}\)};
\node (a3) at (4,4) {\(a_{rn}\)};
\node (a4) at (6,4) {$\ldots$};

\node (10) at (-5,3) {$\ldots$};
\node (12) at (-3,3) {$\ddots$};
\node (12) at (-1,3) {$\ddots$};
\node (12) at (1,3) {\reflectbox{$\ddots$}};
\node (14) at (3,3) {\reflectbox{$\ddots$}};
\node (16) at (5,3) {$\ldots$};

\node (ai+) at (-4,2) {$\ldots$};
\node (ai+) at (-2,2) {\(a_{i,rn-2}\)};
\node (ai+) at (0,2) {\(a_{i+1,rn-1}\)};
\node (aj-) at (2,2) {\(a_{i+2,rn}\)};
\node (ai+) at (4,2) {$\ldots$};

\node (ai+) at (-3,1) {$\hdots$};
\node (aj-) at (-1,1) {\(a_{i,rn-1}\)};
\node (aj) at (1,1) {\(a_{i+1,rn}\)};
\node (aj) at (3,1) {$\hdots$};

\node (ai) at (-2,0) {$\hdots$};
\node (ai+) at (0,0) {\(a_{i,rn}\)};
\node (ai+) at (2,0) {$\ldots$};

\end{tikzpicture}
\end{center}
\caption{Indexing of an infinite periodic frieze, see also Figure \ref{infinitefrieze}.}
\label{fig:infinitefrieze}
\end{figure}

\begin{proposition}\label{prop:growth}
Let $\mathcal F$ be an infinite periodic frieze of (minimal) period n and let $s_r$ and 
$s_{\mathcal F}$ be as in the paragraph above, set $s_0=2$. Then the following holds: 
\begin{enumerate}
\item 
$a_{i,rn+i-1} - a_{i-1,rn+i-2}=s_r$ for all $i\in\mathbb{Z}$; 
\item
$s_{r+1}=s_1s_{r}-s_{r-1}$ for $r>0$;
\item 
$\displaystyle{s_r = s_{\mathcal F}^r+r\sum_{l=1}^{\lfloor r/2\rfloor}(-1)^l\frac{1}{r-l}{r-l\choose l}s_{\mathcal F}^{r-2l}}$
for $r>0$.
\end{enumerate}
\end{proposition}

\begin{proof}
Part (1) is Theorem 2.2 in \cite{bfpt}. Part (2) is Proposition 2.10 in \cite{bfpt}.
Part (3) can be obtained by solving the recursion in Part (2). 
\end{proof}

By Proposition~\ref{prop:growth}(1), the difference between an entry in the $n$-th row 
and the entry directly above it in row $n-2$ in any infinite periodic frieze of period $n$ is constant, 
as is the difference between an entry in row $kn$ 
and the entry above it in row $k n-2$, for 
all $k>1$. The sequence 
$(s_1,s_2,s_3,\dots)$ thus determines the growth of the entries in the frieze. 
By Proposition~\ref{prop:growth}(3), these 
coefficients grow exponentially (\cite[Proposition 4.7]{bfpt}). 

We now concentrate on pairs of  infinite periodic friezes arising from triangulations of an annulus with marked points. 

\begin{definition}
Let $\mathcal T$ be a triangulation of $C_{m,n}$ with associated quiddity sequences 
$q_1=(a_1,\dots, a_m)$ and $q_2=(b_1,\dots, b_n)$ and, infinite periodic friezes $\cF_1$ and 
$\cF_2$. We denote the entries of $\cF_1$ by $a_{ij}$ and the entries of 
$\cF_2$ by $b_{ij}$. Then we define $s_{q_1}:=a_{1,m}-a_{2,m-1}$ and 
$s_{q_2}:=b_{1,n}-b_{2,n-1}$. We call $s_{q_i}$ the {\em growth coefficient of $q_i$}. 
\end{definition}

\begin{remark}
(1)
In order to determine growth coefficients, it is enough to work with 
skeletal quiddity sequences, see~\cite[Theorem 3.1]{bfpt}. 

(2)
In our infinite periodic friezes, we always have $s_{q}> 2$. 
There are infinite periodic friezes with $s_q=2$. In these 
cases, all the coefficients $s_i$ are 2. Such a quiddity sequence generates an arithmetic 
frieze which we have excluded from 
our considerations, see Example~\ref{ex:spiraling} and Remark~\ref{rem:no-arithm}. 

(3)
By definition, the growth coefficient $s_{q}$ of a quiddity sequence $q=(a_1,\dots, a_n)$ 
is equal to one of the 
$s_r$ from above, it is not necessarily equal to the growth coefficient 
$s_1=s_{\cF}$ of the associated frieze, 
as the quiddity sequence may be symmetric, i.e. the repetition of a shorter 
subsequence. We have $s_{q}=s_{\cF}$ if and only 
if the minimal period of $\cF_i$ is equal to $n$, 
i.e. if and only if the minimal period of $\cF_i$ is equal to the length of its quiddity sequence.
\end{remark}

\begin{example}\label{ex:growth}
As an example, consider the leftmost triangulation in Figure~\ref{fig:skeleton-triangulation}. 
The quiddity sequences are $q_1=(1,4,4)$ and $q_2=(3,3)$ and one can compute 
$s_{q_1}=7$, $s_{q_2}=7$: For $q_1$ we have to compute the difference between an entry in the 
3rd non-trivial row and the entry above it in the quiddity row. 
For $q_2$, we compute the difference between an entry 
in the 2nd non-trivial row and an entry above it. 
However, if one views $\cF_2$ as an infinite 1-periodic frieze, one obtains 
$s_{\cF_2}=3$. 
\end{example}

We observe that $s_{q_1}=s_{q_2}$ in Example~\ref{ex:growth}. 
This is no coincidence, as we will see now. 

\begin{proposition}\label{prop:growth-2-boundaries}
Let $\mathcal T$ be a triangulation of an annulus, with associated 
infinite periodic friezes $\cF_1$ and $\cF_2$ and quiddity sequences $q_1$ and $q_2$ respectively. 
Then 
$s_{q_i}=s_{q_i^s}$ for $i=1,2$ and $s_{q_1}=s_{q_2}$. 
\end{proposition}

\begin{proof}
This is Theorem 3.4 in \cite{bfpt}. 
\end{proof}

Let us point out that not every pair of infinite periodic friezes with a common growth coefficient can be 
realized via a triangulation of an annulus. Consider $q_1=q_2=(2,3)$ with growth 
coefficient $s_{q_1}=s_{q_2}=4$. But this pair does not correspond to any triangulation of 
an annulus. Another such example is: $q_1=(4,3,4,3)$ and $q_2=(5,20)$ 
where $s_{q_1}=s_{q_2}=98$. 

%
\subsection{A formula for calculating growth coefficents}\label{sec:formula-matrices}
Consider an arbitrary element \(a_{ij}\) of the frieze in Figure~\ref{fig:infinitefrieze}.
By \cite[Lemma 3.5]{bfpt}, building on work by \cite{br}, we can find the entry \(a_{ij}\) of the infinite periodic frieze by calculating a determinant:
\begin{equation}\label{eq:determinant}
a_{ij}=\det
\begin{tikzpicture}[baseline=(current bounding box.center)]
\matrix (m) [matrix of math nodes, nodes in empty cells, left delimiter={(}, right delimiter={)}]
{
a_i 		& 1 		& 0 		& 	& 	& 0 		\\
1 		& a_{i+1} 	& 1 		& 	& 		& 	\\
0		& 1		& 	& 	& 	& 	\\
 	& 	& 	& 	& 1 		& 0 		\\
 	& 		& 	& 1 		& a_{j-1} 	& 1		\\
0		& 	& 	& 0		& 1 		& a_j	\\
};
\draw [loosely dotted, thick] (m-2-2) -- (m-5-5);
\draw [loosely dotted, thick] (m-2-3) -- (m-4-5);
\draw [loosely dotted, thick] (m-3-2) -- (m-5-4);
\draw [loosely dotted, thick] (m-1-3) -- (m-4-6);
\draw [loosely dotted, thick] (m-3-1) -- (m-6-4);
\draw [loosely dotted, thick] (m-1-3) -- (m-1-6);
\draw [loosely dotted, thick] (m-3-1) -- (m-6-1);
\draw [loosely dotted, thick] (m-1-6) -- (m-4-6);
\draw [loosely dotted, thick] (m-6-1) -- (m-6-4);
\end{tikzpicture}
\end{equation}

\begin{definition}
Consider a finite ordered set of integers \(S=\{i, i+1, \ldots, j-1, j\}\). A \emph{pair-excluding subset} \(I\subseteq S\) is a subset obtained by removing zero or more disjoint pairs of consecutive integers from \(S\).

A \emph{cyclic pair-excluding subset} \(J\subseteq S\) is a subset obtained by removing zero or more disjoint pairs of consecutive integers from \(S\), when the first and the last element of \(S\) are also considered consecutive.
\end{definition}
Note that any pair-excluding subset is also cyclic pair-excluding. The empty set is a pair-excluding subset if and only if the cardinality of \(S\) is even.
\begin{example}
The pair-excluding subsets of \(\{1,2,3,4,5\}\) are
\[
\begin{array}{llll}
\{1, 2, 3, 4, 5\},\\
\{1,2,3\}, &\{1,2,5\},&\{1,4,5\},&\{3,4,5\},\\
\{1\}, & \{3\}, & \text{and }&\{5\}\\
\end{array}
\]
The cyclic pair-excluding subsets are the above, together with the sets
\[
\{2,3,4\}, \{2\} \text{ and }\{4\}.
\]
Similarly, the pair-excluding subsets of \(\{1, 2, 3, 4, 5, 6\}\) are 
\[
\begin{array}{llllll}
\{1,2,3,4,5,6\},\\
\{1,2,3,4\}, & 	\{1,2,3,6\}, &	\{1,2,5,6\}, &	\{1,4,5,6\}, &	\{3,4,5,6\},\\ 
\{1,2\}, & 			\{1,4\}, &		\{1,6\},&		\{3,4\}, &		\{3,6\}, & 		\{5,6\}\\
\text{and } \emptyset.
\end{array}
\]
The cyclic pair-excluding subsets are the sets above, together with the sets
\[
\{2,3,4, 5\}, \{2,3\}, \{2,5\} \text{ and } \{4, 5\}.
\]
\end{example}

\begin{theorem}\label{thm:continuant} 
With the notation above, 
\[
a_{ij}=\sum_{\substack{I\subseteq \{i, \ldots, j\}\\ \text{pair-excluding}}}(-1)^{\ell_I}\prod_{k\in I}a_k
\]
where \(\ell_I\) is the number of pairs that were excluded from \(\{i, \ldots, j\}\) to create \(I\); in 
other words \(\ell_I=\frac{j-i+1-|I|}{2}\).
\end{theorem}
\begin{proof}
The theorem follows from using \S 544-546 of \cite{muir} on formula \ref{eq:determinant}.
\end{proof}

\begin{example}
For \(i=j\), the formula simply reads $a_{ii}=a_{ii}$, and for \(j=i+1\), it says 
$a_{i,i+1}=a_{ii}a_{i+1,i+1}-1$, 
where the term $-1$ appears due to the empty set being a pair-excluding subset when the original 
set is of even length.
\noindent For \(i=1\), \(j=5\) the formula is as follows:
\begin{align*}
a_{15}=&a_1a_2a_3a_4a_5\\
& - a_1a_2a_3 -a_1a_2a_5-a_1a_4a_5-a_3a_4a_5\\
& + a_1 + a_3 + a_5 
\end{align*}
For \(i=1\), \(j=6\) the formula is as follows:
\begin{align*}
a_{16}=&a_1a_2a_3a_4a_5a_6 \\
& - a_1a_2a_3a_4 - a_1a_2a_3a_6 - a_1a_2a_5a_6 - a_1a_4a_5a_6 - a_3a_4a_5a_6 \\
&+ a_1a_2 + a_1a_4 + a_1a_6+a_3a_4 +a_3a_6+a_5a_6\\
&-1.
\end{align*}
\end{example}

\begin{corollary}
\label{cor:frieze-formula}
Consider an infinite periodic frieze with quiddity sequence \(q=(a_1, \ldots, a_n)\). 
The growth coefficient of this frieze is given by
\[s_q=\left(\sum_{\substack{I\subseteq\{1, \ldots, n\} \\ \text{cyclical} \\ \text{pair-excluding} }}(-1)^{\ell_I} \prod_{k\in I}a_k\right)+\delta_n\]
Here \(\ell_I\) is the number of pairs that were excluded from \(\{1, \cdots, n\}\) to create \(I\); in other words \(\ell_I=\frac{n-|I|}{2}\). Furthermore, 
\[\delta_n=\begin{cases}
0 &\text{ if } n \text{ is odd},\\
1 &\text{ if } n \text{ is divisible by } 4, \\
-1 &\text{ if } n \text{ otherwise.}
\end{cases}\]
\end{corollary}
\begin{proof}
We know that the growth coefficient is obtained by taking an entry in the \(n\)th row of the frieze, and subtracting the entry directly above it, in row \(n-2\). The growth coefficient is independent of the choice of entry, so we may write \(s_q=a_{1,n}-a_{2,n-1}\)

Consider the pair-excluding subsets of \(\{2, \ldots, n-1\}\). These are precisely the subsets of \(\{1, \ldots, n\}\) which are cyclic pair-excluding, but not pair-excluding.
Hence we get
\begin{align*}
s_q=&a_{1,n}-a_{2,n-1}\\
=&\left(\sum_{\substack{I\subseteq \{1, \ldots, n\}\\ \text{pair-excluding}}}(-1)^{\ell_I}\prod_{k\in I}a_k\right) 
-\left(\sum_{\substack{I\subseteq \{2, \ldots,n-1\}\\ \text{pair-excluding}}}(-1)^{\ell_I}\prod_{k\in I}a_k\right)\\
=&\sum_{\substack{I\subseteq\{1, \ldots, n\} \\ \text{cyclical} \\ \text{pair-excluding} }}(-1)^{\ell_I} \prod_{k\in I}a_k+\delta_n.
\end{align*}
For \(n\) even, we have a term \(\pm 1\) in the expression for \(a_{2,n-1}\) which is recovered by the inclusion of \(\delta_n\) in the formula; this term does not occur for odd \(n\).
\end{proof}
\begin{example} \label{ex:formula}
Consider the infinite periodic frieze given by the quiddity sequence \(q=(2,3,4,2,4)\) as illustrated in Figure \ref{fig:5-frieze}. By Corollary \ref{cor:frieze-formula}, we calculate the growth coefficient to be 
\begin{align*}
s_q=&2\cdot 3 \cdot 4 \cdot 2\cdot 4 \\
&- (2\cdot 3\cdot 4 + 2\cdot 3\cdot 4+2\cdot 2\cdot 4 - 4\cdot 2 \cdot 4 - 2\cdot 3\cdot 4) \\
&+(2+3+4+2+4)\\
=&87,
\end{align*}
which is also the growth coefficient we read off the frieze in Figure \ref{fig:5-frieze}.

\begin{figure}[h]
\begin{tikzpicture}
 \matrix[matrix of math nodes, nodes in empty cells] (m) at (0,0)
 {
 0 && 0 && 0 && 0 && 0 && 0 && 0 && 0 && 0 && 0\\
 & 1 && 1 && 1 && 1 && 1 && 1 && 1 && 1 && 1 && 1 \\
 2 && 3 && 4 && 2 && 4 && 2 && 3 && 4 && 2 && 4\\
 & 5 && 11 && 7 && 7 && 7 && 5 && 11 && 7 && 7 && 7 \\
 17 && 18 && 19 && 24 && 12 && 17 && 18 && 19 && 24 && 12\\
 & 62 && 31 && 65 && 41 && 29 && 62 && 31 && 65 && 41 && 29 \\
 104 && 105 && 106 && 111 && 99 && 104 && 105 && 106 && 111 && 99 \\
 & \vdots && \vdots && \vdots && \vdots && \vdots && \vdots && \vdots && \vdots && \vdots && \vdots \\
 };
\end{tikzpicture}
\caption{The frieze with quiddity sequence \((2,3,4,2,4)\).} 
\label{fig:5-frieze}
\end{figure}

\end{example}

%
\subsection{Cluster category for an infinite periodic frieze} \label{sec:frieze-to-category} 

We consider a cluster category 
$\cC=D^b(\C Q)/\tau^{-1}[1]$ 
of type $\widetilde{A}$ associated to a non-oriented cyclic quiver $Q$ with $n+m$ vertices, 
where $m$ arrows are of the form $i\to i+1$ and $n$ arrows of the form $i\to i-1$, 
as defined in~\cite{bmrrt}. 
The Auslander--Reiten quiver (AR quiver) 
of this cluster category has two tubes, one of rank $n$ and one of rank $m$, 
see~\cite{bt} for details. 
We use the specialized Caldero--Chapoton (CC) map \cite{cc} (see \cite{p} for the general setting) with respect to the algebra $\C Q$ 
to associate infinite periodic friezes to the tubes of the AR quiver of $\cC$. 
In this case, the CC map sends the shifted projectives to 1. 
All other indecomposables can be viewed as $\C Q$-modules. 
The CC map sends such a module $M$ to $\sum_{\underline{e}}\chi(Gr_{\underline{e}}M)$ where the sum is over submodules of $M$ with dimension vector $e$ and $\chi$ is the Euler-Poincar\'e characteristic of the complex Grassmannian. 
For rigid modules in the tube, this is the number of submodules. For non-rigid modules, i.e. modules 
in a tube of rank $n$, which are at least $n$ steps away from the mouth of the tube, 
the Euler-Poincar\'e characteristic can be computed by taking into account 
simples which occur with higher multiplicities. 
Indeed, for any module M in a tube
there is a representation of $M$ by an arc, see e.g.~\cite{warkentin,abcp,BM-tube}.
In the same setting, the authors in \cite{msw} construct a planar graph to each arc, called 
\emph{snake graph}, and use cardinalities of certain sets in this graph, namely \emph{perfect matchings}, 
to compute the Euler-Poincar\'e characteristic $\chi(Gr_{\underline{e}}M)$ of a fixed dimension vector 
$\underline{e}$, see~\cite[Theorem 13.1]{msw}. By establishing a bijection between the lattice structure 
of perfect matchings of a snake graph and the \emph{extended} submodule structure of $M$ (this 
corresponds to the submodule lattice of $M'$ where $M'$ is obtained by 
considering the simple composition factors of $M$ with their multiplicities), 
\cite[Corollary 3.10]{cs} establishes a 
formula for $\chi(Gr_{\underline{e}}M)$ in terms of submodules of $M$.

We denote the specialized CC map by $s$. Thus 
if $M$ is an indecomposable object of $\cC$ which is not a shifted projective 
and which is at most $n-2$ steps away form the mouth of the tube
(
where the mouth itself is 
considered to be 0 steps away
), then 
$s(M)$ is 
simply
the number of submodules of $M$ and if $M$ is a shifted projective, 
$s(M)=1$. 
If $M$ is
 at least than $n-1$ steps away from the mouth of a tube of rank $n$
 , 
then $s$ is computed by counting the submodules of $M'$ where $M'$ is obtained by relabelling the simples that occur with higher multiplicities in $M$.
In type $\widetilde{A}$ almost split sequences have either one or two middle terms. After applying the specialized CC map to the tubes of the AR quiver, the entries satisfy the determinant rule to form two infinite periodic friezes, one for each tube (similar as in Section 5 of \cite{cc}). 
Namely, if $\tau M \to B \to M$ is an Auslander-Reiten triangle in the cluster category $\cC$, then applying the CC map by \cite[Proposition 2.2]{ad}, we obtain $s(\tau M)s(M) - s(B) =1$ where $s$ is the specialized CC-map. 


%
\subsection{Module-theoretic interpretation} \label{sec:modules-growth}

In this section, we associate a cluster tilting object in a cluster category of type $\widetilde{A}$ to an infinite periodic frieze. 
Let $q$ be the quiddity sequence of a skeletal infinite 
periodic frieze $\cF$. This gives rise to a skeletal triangulation $\cT=\cT(q)$
of an annulus with $n$ marked points 
on one boundary component (say $B_1$) of the annulus, by Corollary~\ref{cor:quiddity to triangulation}. 
Let $m$ be the number of marked points on $B_2$ (see Section~\ref{sec:tria-frieze}). 
Then $\cT$ is a triangulation of the annulus
$C_{n,m}$. Let $\cC=\cC_{n,m}$ be the category from Section~\ref{sec:frieze-to-category}. 
The indecomposable objects of 
this category except the ones that come from the homogeneous tubes are known to be in bijection with arcs in 
$C_{n,m}$, and triangulations of $C_{n,m}$ correspond to cluster-tilting objects for 
$\cC$, as described in~\cite{BM-tube, warkentin}. 
We claim that one of the infinite periodic friezes coming from the tubes is $\cF$. Note that 
the indecomposable objects in the tubes correspond to the peripheral arcs on the boundaries $B_1$ 
and $B_2$ respectively. We restrict to the tube arising from $B_1$. We use the following notation to 
label indecomposables in that tube. 
Note that we follow the frieze notation here and draw the 
vertices at the mouth of the tube on the top line. Going down in the quiver thus corresponds to 
injections.

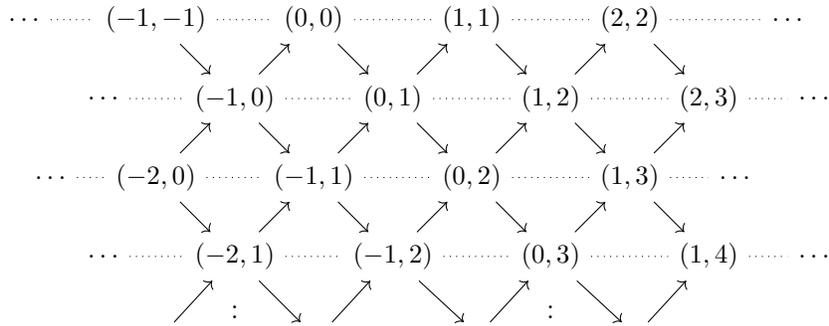
\begin{figure}[ht!]
\begin{center}
\begin{tikzpicture}[scale=.7, inner sep=4pt]

\node (l1) at (-7.5,5) {$\ldots$};
\node (-1-1) at (-5,5) {($-1,-1$)};
\node (00) at (-2,5) {($0,0$)};
\node (11) at (1,5) {($1,1$)};
\node (22) at (4,5) {($2,2$)};
\node (r1) at (7,5) {$\ldots$};

\node (l2) at (-6,3.5) {$\ldots$};
\node (-10) at (-3.5,3.5) {\((-1, 0)\)};
\node (01) at (-0.5,3.5) {($0,1$)};
\node (12) at (2.5,3.5) {($1,2$)};
\node (23) at (5.5,3.5) {($2,3$)};
\node (r2) at (7.5,3.5) {$\ldots$};

\node (l3) at (-7,2) {$\ldots$};
\node (-20) at (-5,2) {($-2,0$)};
\node (-11) at (-2,2) {($-1,1$)};
\node (02) at (1,2) {($0,2$)};
\node (13) at (4,2) {($1,3$)};
\node (r3) at (6,2) {$\ldots$};

\node (l4) at (-6,0.5) {$\ldots$};
\node (-21) at (-3.5,0.5) {($-2,1$)};
\node (-12) at (-0.5,0.5) {($-1,2$)};
\node (03) at (2.5,0.5) {($0,3$)};
\node (14) at (5.5,0.5) {($1,4$)};
\node (r4) at (7.5,0.5) {$\ldots$};

\node (ai) at (-3.5,-.5) {$\vdots$};
\node (ai+) at (2.5,-.5) {$\vdots$};

\draw[->] (-1-1) to (-10);
\draw[->] (-10) to (00);
\draw[->] (00) to (01);
\draw[->] (01) to (11);
\draw[->] (11) to (12);
\draw[->] (12) to (22) ;
\draw[->] (22) to (23);

\draw[->] (-20) to (-10);
\draw[->] (-10) to (-11) ;
\draw[->] (-11) to (01);
\draw[->] (01) to (02);
\draw[->] (02) to (12) ;
\draw[->] (12) to (13);
\draw[->] (13) to (23);

\draw[->] (-20) to (-21);
\draw[->] (-21) to (-11) ;
\draw[->] (-11) to (-12);
\draw[->] (-12) to (02);
\draw[->] (02) to (03) ;
\draw[->] (03) to (13);
\draw[->] (13) to (14);

\draw[->] (-4.65,-.75) to (-21);
\draw[->] (-21) to (-2.25,-.75);
\draw[->] (-1.65,-.75) to (-12);
\draw[<-] (.85,-.75) to (-12);
\draw[->] (1.35,-.75) to (03);
\draw[->] (03) to (3.85,-.75);
\draw[->] (4.35,-.75) to (14);
\draw[dotted] (l1) to (-1-1) to (00);
\draw[dotted] (11) to (00);
\draw[dotted] (11) to (22) to (r1);
\draw[dotted] (l2) to (-10) to (01) to (12) to (23) to (r2);
\draw[dotted] (l3) to (-20) to (-11) to (02) to (13) to (r3);
\draw[dotted] (l4) to (-21) to (-12) to (03) to (14) to (r4);

\end{tikzpicture}
\end{center}
\end{figure}

%
We write $M_{ij}$ for the indecomposable module corresponding to the vertex $(i,j)$ with 
indices reduced modulo $n$ (Figure~\ref{fig:infinitefrieze-modules}). The dotted lines indicate the Auslander--Reiten translation $\tau$ in the category, which sends the module 
$M_{ij}$ to the module $M_{i-1,j-1}$. Let $M_i = M_{ii}$. 
Note that the indexing of vertices/modules follows the indexing we use for friezes. 
Under the correspondence between indecomposables and arcs in $C_{n,m}$, the indecomposable 
objects in this tube correspond to the peripheral arcs based at $B_1$. In particular, 
the $M_i$'s correspond to the peripheral arcs of $B_1$ joining the marked points $i$ and $i+2$.

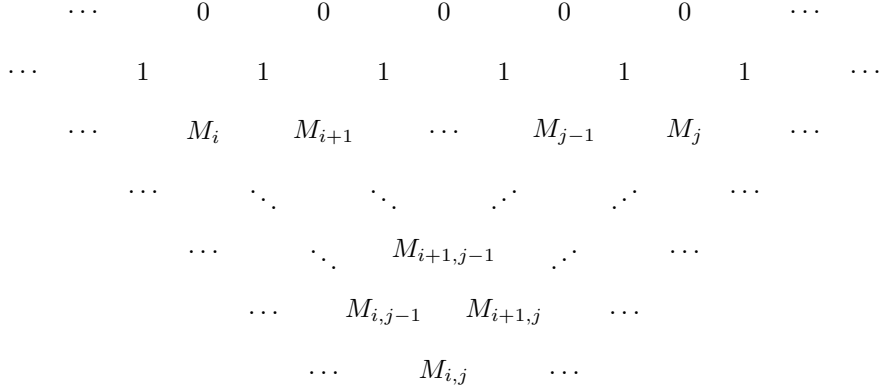
\begin{figure}[ht!]
\begin{center}
\begin{tikzpicture}[scale=.8, inner sep=4pt]
\node (00) at (-6,6) {$\ldots$};
\node (01) at (-4,6) {0};
\node (02) at (-2,6) {0};
\node (ai+) at (0,6) {0};
\node (03) at (2,6) {0};
\node (04) at (4,6) {0};
\node (05) at (6,6) {$\ldots$};
\node (10) at (-7,5) {$\ldots$};
\node (11) at (-5,5) {1};
\node (12) at (-3,5) {1};
\node (12) at (-1,5) {1};
\node (12) at (1,5) {1};
\node (14) at (3,5) {1};
\node (15) at (5,5) {1};
\node (16) at (7,5) {$\ldots$};

\node (a-2) at (-6,4) {$\ldots$};
\node (a-1) at (-4,4) {$M_i$};
\node (a0) at (-2,4) {$M_{i+1}$};
\node (a1) at (0,4) {$\hdots$};
\node (a2) at (2,4) {$M_{j-1}$};
\node (a3) at (4,4) {$M_j$};
\node (a4) at (6,4) {$\ldots$};

\node (10) at (-5,3) {$\ldots$};
\node (12) at (-3,3) {$\ddots$};
\node (12) at (-1,3) {$\ddots$};
\node (12) at (1,3) {\reflectbox{$\ddots$}};
\node (14) at (3,3) {\reflectbox{$\ddots$}};
\node (16) at (5,3) {$\ldots$};

\node (ai+) at (-4,2) {$\ldots$};
\node (ai+) at (-2,2) {$\ddots$};
\node (ai+) at (0,2) {$M_{i+1, j-1}$};
\node (aj-) at (2,2) {\reflectbox{$\ddots$}};
\node (ai+) at (4,2) {$\ldots$};

\node (ai+) at (-3,1) {$\hdots$};
\node (aj-) at (-1,1) {$M_{i, j-1}$};
\node (aj) at (1,1) {$M_{i+1, j}$};
\node (aj) at (3,1) {$\hdots$};

\node (ai) at (-2,0) {$\hdots$};
\node (ai+) at (0,0) {$M_{i,j}$};
\node (ai+) at (2,0) {$\ldots$};

\end{tikzpicture}
\end{center}
\caption{A tube of the AR quiver with modules $M=M_{i,j}$ and $\widetilde{M}=M_{i+1, j-1}$, and modules $M_i, \ldots, M_j$ at its mouth, where $j=i+t-1$.} 
\label{fig:infinitefrieze-modules}
\end{figure}

In order to link the indecomposable objects with the friezes, we use the specialized CC map 
$s$ defined above. We can then use \cite[Section 0.1]{hj} to obtain part (1) of the following: 

\begin{proposition}\label{prop:entries-via-s}
Following the notation above and with $M:=M_{i,i+t-1}$, for $t\ge 1$, \\
(1) $s(M_i)=a_{i}$. \\
(2) $s(M)= a_{i,i+t-1}$
\end{proposition}
\begin{proof}
(2) follows from (1), using the fact that $s$ is multiplicative on short exact sequences, 
see for example~\cite[Section 0.1]{hj}. 
\end{proof}
The modules $M_{i}$ form the {\em mouth} of the category. All modules in 
a given $\tau$-orbit are of the form $M_{i,i+n-1}$ for some $n>0$. We say that module 
$M_{i,i+n-1}$ is at level $n$. 

Let $M$ be an indecomposable in $\cC$ at level $n$. The {\em wing of $M$}, denoted by $\cW(M)$, 
consists of the indecomposable modules which are submodules, quotients or subquotients of $M$.
In the AR quiver, these are the indecomposables positioned in the triangle in $\Gamma$ whose apex is $M$.
This means that if $M=M_{i,i+n-1}$, then $\cW(M)$ is given by $\{M_{u,v}\mid i\le u\le v\le i+n-1\}$. 


Then $M_i,\dots, M_{i+n-1}$ are the modules at the mouth of the wing $\cW(M)$. 
Note that $M$ is obtained through iterated extensions from 
$M_i,\dots, M_{i+n-1}$. 
Now assume $n\ge 3$ and 
let $\widetilde{M}$ be $M_{i+1, i+n-2}$, so that $\cW(\widetilde{M})=\{M_{u,v}\mid i+1\le u\le v\le i+n-2\}$. 
Furthermore, define $N:=M_i\oplus M_{i+1}\oplus \dots \oplus M_{i+n-1}$.

In the following, we will quotient out direct sums of cyclic consecutive pairs of 
indecomposables from $N$: for $i \leq j \leq i+n-1$, let 
\[
N_j:= N/(M_{j}\oplus M_{j+1}).
\] 
In particular, 
$N_{i+n-1}= N/(M_{i+n-1}\oplus M_{i})$ since we reduce modulo $n$. 
Similarly, for $i \leq j_1 < j_2-1 \le i +n-2$, 
\[
N_{j_1,j_2}:= N/\left( M_{j_1}\oplus M_{j_1+1} \oplus M_{j_2}\oplus M_{j_2+1}\right)
\] 
and for $j_1,\dots, j_k$ with $k\ge 2$, $j_1< j_2-1 <\ldots < j_k-(k-1)\le i+n-k$, 
\[
N_{j_1,j_2,\ldots,j_k}:=N/\left(M_{j_1}\oplus M_{j+1+1}\oplus M_{j_2}\oplus M_{j_2+1} \oplus 
\dots \oplus M_{i_k}\oplus M_{i_k+1}\right)
\]
In particular, for $k=2$ and $j_2=i+n-1$, this gives 
$N_{j_1,i+n-1}=N/(M_i\oplus M_{j_1}\oplus M_{j_1+1}\oplus M_{i+n-1})$. 
Analogously, modules $N_{i_1,\dots, i_k}$ are defined, for $k\le n/2$.

Let $\{r_1,\dots, r_n\}=\{1,2,\dots, n\}$ (not necessarily ordered). We will use the following 
two properties in the proof of the theorem below: 
\begin{align}\label{eq:quotient}
 s(N/(M_{r_1}\oplus\cdots \oplus M_{r_m}))=s(M_{r_{m+1}}\oplus\cdots\oplus M_{r_n}) 
\end{align}
\begin{align} \label{eq:direct-sum}
s(M_{r_{m+1}}\oplus \cdots \oplus M_{r_t})=s(M_{r_{m+1}})\cdots s(M_{r_n})
\end{align}

\begin{theorem}\label{thm:repth}
Let $M$ be an indecomposable of $\cC$ at level $t\ge 3$, 
let $\widetilde{M}$ and $M_i,\dots, M_{i+t-1}$ be as above. \\
If $t=2k+1$,
\[s(M) -s(\widetilde{M}) = s(N) - \sum_{j=i}^{i+t-1} s(N_j) + \sum_{\substack{j_1, j_2=i \\ j_1< j_2-1}} ^{i+t-1}s(N_{j_1,j_2}) + \dots \pm (-1)^k \sum_{j=i}^{i+t-1} s(M_j). \]
If $t=2k$,
\[s(M) -s(\widetilde{M}) = s(N) - \sum_{j=i}^{i+t-1} s(N_j) + \sum_{\substack{j_1, j_2=i \\ j_1< j_2-1}} ^{i+t-1}s(N_{j_1,j_2}) + \dots \pm (-1)^k \cdot 2 .\]
\end{theorem}

Note that the $(k+1)$st term on the right hand of the expressions is 
\[ 
\sum_{\substack{j_1,\ldots, j_k=i \\ j_1< j_2-1 <\ldots < j_k-(k-1)}} ^{i+t-1}s(N_{j_1,j_2,\ldots,j_k}).
\]

\begin{proof} 
We use the formula from Theorem~\ref{thm:continuant} for $M=M_{i,i+t-1}$ and 
$\widetilde{M}=M_{i+1,i+t-2}$ 
with Proposition~\ref{prop:entries-via-s} (2)
:
\begin{eqnarray*}
s(M) - s(\widetilde{M}) & = & a_{i,i+t-1} - a_{i+1,i+t-2} \\ 
 & = & \sum_{\substack{I\subseteq \{i, \ldots, i+t-1\}\\ \text{pair-excluding}}} (-1)^{\ell_I}\prod_{k\in I}a_k 
 \quad - \quad
 \sum_{\substack{I\subseteq \{i+1, \ldots, i+t-2\}\\ \text{pair-excluding}}} (-1)^{\ell_I}\prod_{k\in I}a_k 
\end{eqnarray*}

Consider the second term on the right hand side: we can view pair-excluding subsets of 
$\{i+1,\dots, i+t-2\}$ as the cyclically pair-excluding subsets of $\{i,\dots, i+t-1\}$ which are not 
pair-excluding subsets of $\{i,\dots, i+t-1\}$. Now observe that for any (cyclically) pair-excluding subset 
$I$, the 
sign of $(-1)^{\ell_I}$ when $I$ is viewed as a subset of $\{i,\dots, i+t-1\}$ is the opposite of the sign of 
$(-1)^{\ell_I}$ when $I$ is viewed as a subset of $\{i+1,\dots, i+t-2\}$. 
Given that the second term on the right is subtracted 
from the first, these signs cancel each other. We can thus add the cyclic pair-excluding 
subsets from the second term to the first term and get the following:

\begin{eqnarray*}
s(M)-s(\widetilde{M}) & = & \sum_{\substack{I\subseteq \{i, \ldots, i+t-1\}\\ \text{cyclic} \\ \text{pair-excluding}}} 
(-1)^{\ell_I}\prod_{k\in I}a_k
\end{eqnarray*}

The claim then follows using the equations (\ref{eq:quotient}) and 
(\ref{eq:direct-sum}) since for any cyclic pair-excluding subset $I$ of $\{i,\dots, i+t-1\}$, 
\[
\prod_{k\in I} a_k=\prod_{i\in I} s(M_i)=s(\bigoplus_{i\in I} M_i)=s(N/\bigoplus_{i\notin I}M_i)
\]
where the last term is equal to $s(N_{j_1,\dots, j_r})$ for 
$\{i,i+1,\dots, i+t-1\}\setminus I=\{j_1,j_1+1,j_2,j_2+1,\dots, j_r,j_r+1\}$. 
\end{proof}

We can use the theorem to give the growth coefficient of a quiddity sequence a module-theoretic 
interpretation. For a given skeletal quiddity sequence $(a_1,\dots, a_n)$ of an infinite periodic frieze, 
let $\cT$ be the associated triangulation of $C_{n,m}$ as at the beginning of this section. 
Then take the cluster category $\cC_{n,m}$ as in Section~\ref{sec:frieze-to-category}
and let $s$ be the specialized CC map. 

\begin{corollary}\label{cor:growth-cc-map}
Let $q=(a_1,\dots, a_n)$ be a quiddity sequence, let $\cB$ be the associated rank $n$ tube as 
above. Let 
$M=M_{i,i+n-1}$ be any indecomposable in $\cB$ at level $n$ and 
$\widetilde{M}=M_{i+1,i+n-2}$. Then we have 
\[
s(M)-s(\widetilde{M})=s_{q}. 
\]
\end{corollary}


\noindent {\bf Acknowledgements:} The authors would like to thank Chelsea Walton, Georgia Benkart, 
Eleonore Faber, Ellen Kirkman and other organizers for organizing WINART2 
at University of Leeds, where this project started. 

\noindent {\bf Funding:} This work was supported by the FWF [grant numbers P30549-N26, DK1230]; a Royal Society Wolfson Fellowship; the Engineering and Physical Sciences Research Council [grant number EP/P016014/1]; the Alexander von Humboldt Fellowship; the Institute for Advanced Study (Princeton) and the Max Planck Institute for Mathematics (Bonn).

\end{document}